\documentclass[12pt]{amsart}
\usepackage{amsmath,amssymb,enumerate}
\usepackage{epsf,epsfig,amsfonts,graphicx,a4wide}  
\usepackage[dvipsnames]{xcolor}
 \usepackage[applemac]{inputenc} 
 \newcommand{\weg}[1]{}
 
 \date{November 19, 2012}   
  
 \numberwithin{equation}{section}   
\newcommand{\ds}{\displaystyle}
\newcommand{\tq}{\, \big| \, }
\renewcommand{\r}{\mathbb{R}}
\newcommand{\R}{\mathbb{R}}
\DeclareMathOperator{\Vol}{Vol}%
\DeclareMathOperator{\area}{Area}%
\DeclareMathOperator{\const}{const}%
\DeclareMathOperator{\jac}{Jac}%
\DeclareMathOperator{\Aut}{Aut}%
\DeclareMathOperator{\End}{End}%
\newcommand{\W}{\operatorname{W}}
\newtheorem{theorem}{\rm\bf Theorem}[section]
\newtheorem{proposition}[theorem]{\rm\bf Proposition}
\newtheorem{lemma}[theorem]{\rm\bf Lemma}
\newtheorem{corollary}[theorem]{\rm\bf Corollary}
\newtheorem*{claim}{\rm\bf Claim}
\theoremstyle{definition}
\newtheorem{definition}[theorem]{\rm\bf Definition}
\newtheorem{remark}[theorem]{\rm\bf Remark}
\newtheorem*{Remark}{\rm\bf Remark}

\title[The Binet-Legendre Metric]{The Binet-Legendre Metric  \\ in  Finsler Geometry}

\author{Vladimir S. Matveev} 
\address{Mathematisches Institut, Friedrich-Schiller Universit\"at Jena\\
07737 Jena, Germany}  
\email{vladimir.matveev@uni-jena.de}

\author{Marc Troyanov} 
\address{Section de Math{\'e}matiques,  
\'Ecole Polytechnique F{\'e}derale de Lausanne, station 8,
1015 Lausanne - Switzerland}
\email{marc.troyanov@epfl.ch}

\begin{document}

\hfill{\small  \it Published in Geometry \& Topology 16 (2012) 2135–2170} \vspace{+1.6cm} 

\begin{abstract} For every  Finsler metric $F$ we associate  a Riemannian metric $g_F$ (called the \emph{Binet-Legendre} metric).  The Riemannian metric $g_F$  behaves nicely under conformal  deformation of the Finsler metric $F$, which makes it   a powerful tool in Finsler geometry.  We illustrate that by solving a number of named Finslerian geometric problems. We also generalize and give new and shorter proofs of a number of known results. In particular we answer a question of M. Matsumoto about local conformal mapping between two Minkowski spaces, we describe all possible conformal self maps and all self similarities on a Finsler manifold.  We also classify all compact conformally flat Finsler manifolds. We  solve a conjecture of S. Deng and Z. Hou on the Berwaldian character of locally symmetric Finsler spaces, \ and extend a classic result by H.C.Wang about the maximal dimension of the isometry groups of Finsler manifolds to manifolds of all  dimensions.
 
Most proofs in this paper  go along  the following scheme: using the correspondence $F \mapsto g_F$ 
we reduce the Finslerian problem to a similar problem for the Binet-Legendre metric, which is easier and is already   solved  in most cases we consider. The  solution of the Riemannian problem provides us with the additional information that helps to solve the initial Finslerian problem.

Our methods apply even in the absence of the strong convexity assumption usually assumed in Finsler geometry. The smoothness hypothesis can also be replaced by the weaker notion of  \emph{partial smoothness}, a notion we introduce in the paper. Our results apply therefore to a vast class of Finsler metrics not usually considered in the Finsler literature.

\medskip
 

\medskip

\end{abstract}

\maketitle      

\section{Introduction}  \label{1}

In the present paper, a  \emph{Finsler metric}  on a smooth manifold  $M$  is   a  continuous 
function $F:TM\to [0,\infty)$ such that for every point $x\in M$ the  restriction $F_x = F_{|T_xM}$ on the tangent space at $x$ is a  
\emph{Minkowski norm}, that is $F_x$  is positively homogenous and convex and it vanishes only at $\xi = 0$: 
\begin{enumerate}[ \ (a)]
  \item $F_x(\lambda \cdot \xi) = \lambda \cdot   F_0 (\xi) $ for any $\lambda \geq 0$.
   \item $F_x (\xi+ \eta ) \le F_{0} (\xi) + F_{0} (\eta)$.
   \item $F_x (\xi)= 0 $ \  $ \Rightarrow$ \  $\xi=0$.  
  \end{enumerate}  
Observe that $F_x$ is a norm in the usual sense if and only if it is symmetric: $F_x(-\xi) = F_x(\xi)$.
The Finsler metric is said to be \emph{of class $C^k$} if  the restriction of $F$ to the slit tangent bundle
$TM^0 = TM\setminus (\text{the zero section})$ is a function of class  $C^k$.
Note that it is customary  in Finsler geometry to require the Finsler metric to be of class  $C^2$ and  \emph{strongly convex}, 
that is the Hessian of the restriction of $F^2$ to $T_xM\setminus \{0 \}$ is assumed  to be positive definite for any 
$x \in M$. However we shall avoid these hypothesis  as they exclude from
the theory some interesting and important examples.  

\smallskip

Our goal in this paper is to solve a number of open problems in Finsler geometry by reducing them to problems in Riemannian Geometry. 
The method is to associate a natural Riemannian metric $g_F$ to a given Finsler metric $F$ on a smooth manifold $M$.   
We use a construction which comes from classical mechanics  and convex geometry: 
we first define the scalar product  $g_F^*$  on the cotangent  space $T^*_x M$ of a given point to be a normalized $L^2$ scalar product of the restrictions of  
$\theta, \phi\in T^*_x M$  to the the unit ball  $\Omega_x = \{\xi\in  T_xM\mid F(x, \xi)\le 1\} \subset T_xM$, that is
\begin{equation}\label{eq.defBLdual}
   g^*_F(\theta, \varphi) =  
   \frac{(n+2)}{\lambda(\Omega_x)}\int_{\Omega_x} \left( \theta(\eta)\cdot \varphi(\eta)\right) \,  d\lambda (\eta).
\end{equation}
where $d\lambda$  is an arbitrary  linear volume form on $T_xM$  and $\lambda(\Omega_x)$ is the volume of $\Omega_x$ with respect to  $d\lambda$.  It is clear that 
$g^*_F$ is a scalar product and that it is independent of the choice of $d\lambda$.

\begin{definition}
The  \emph{Binet-Legendre metric} associated to a Finsler metric $F$ on a smooth manifold $M$ is the Riemannian metric $g_F$ dual to the  scalar product defined above:
\begin{equation}\label{eq.duality1}
g_F(\xi,\eta) =  g^*_F(\xi^{\flat},\eta^{\flat})
\end{equation}
for any $\xi, \eta \in TM$, where $\xi^{\flat} \in T_x^*M$ is defined as  $g_F^*(\xi^{\flat}, \theta) =  \theta(\xi)$ for all $\theta\in T^*_xM$.
\end{definition}

 The motivation for the name \emph{Binet-Legendre} comes from the fact that the unit ball  of $g^*_F$ in the cotangent space $T^*_xM$ is the so called \emph{Binnet ellipsoid}
of the convex body $\Omega_x \subset T_xM$, while the  the unit ball  of $g^*_F$ in the tangent space $T^*_xM$ is  up to a scaling constant, see Remark \ref{rem5})
its \emph{Legendre ellipsoid}. These ellipsoids have their roots in the 19th century description of rigid bodies dynamics and have more recently been a subject  studied in
convex geometry, see for example  \cite{Milman}.

It seems that the Binet-Legendre metric has so far not attracted the attention it deserves in Finsler geometry.
It appears  under the name \emph{osculating Riemannian metric} 
in the paper \cite{Centore} by P. Centore where it is proven that the Hausdorff measure on a Finsler manifold is greater or equal to the Binet-Legendre Riemannian volume form.
We did not find any other published work on the Binet-Legendre metric in Finsler geometry, although the idea is  probably known to the experts.

\medskip

The Binet-Legendre metric  is  one among many possible  ways to construct a Riemannian metric on a Finsler manifolds; its importance lies in the fact that it satisfies the following natural functorial properties:

\begin{theorem}\label{th.BLproperties1}
The  Binet-Legendre metric $g_F$ associated to the  Finsler manifold $(M,F)$ satisfies the following properties:
 \begin{enumerate}[a)]
  \item If  $F$  of class $C^k$, then so is $g_F$.
  \item If  $F$  is Riemannian, i.e., if $F(x, \xi)= \sqrt{g_{x}(\xi, \xi)}$ for  some Riemannian metric $g$, then $g_F= g$.
  \item If $A\in \Aut (TM)$ is a $C^k$-field of automorphisms of the tangent bundle of $M$, then $g_{A^*F}=A^*g_F$.
  \item  \label{(d)} If $F_1(x,\xi)= \lambda(x)\cdot F_2(x,\xi)$ for some function $\lambda : M \to \r_+$, then  $g_{F_1}= \lambda^2\cdot g_{F_2}$.
  \item 
  If \ $\frac{1}{\lambda}\cdot F_1 \leq F_2 \leq\lambda  \cdot F_1$ for some  function $\lambda : M \to \r_+$, then  
  $$\frac{1}{\lambda ^{2n}}\cdot g_{F_1} \leq g_{F_2} \leq \lambda ^{2n} \cdot g_{F_1}.$$
\end{enumerate}
\end{theorem}
 
\begin{proof}  
The first property  is Theorem \ref{th.smooth} below,   which is actually a stronger result, combined with example (a) in section \ref{sec.partsmooth}.  Properties (b)--(e)  are essentially known facts  about the Legendre ellipsoid. For the convenience of the reader  we prove them in the appendix, see Proposition \ref{prop.binet2}.
\end{proof}

This theorem immediately implies the following consequences:
\begin{itemize} 
   \item[$\circ$]  If two Finsler metrics $F _1$ and $F_2$ are conformally equivalent, i.e., if $F_1(x,\xi)= \lambda(x)\cdot F_2(x,\xi)$ for some function $\lambda : M \to \r$ , then the corresponding Riemannian metrics are also conformally equivalent with essentially the same conformal factor: $g_{F_1}= \lambda^2\cdot g_{F_2}$.
    In particular every conformal transformation of the Finsler metrics is conformal for the Binet-Legendre metric.    
   \item[$\circ$]  The Binet-Legendre construction is $C^0$-stable: if $F_1$ and $F_2$ are $C^0$-close, then so are  $g_{F_1}$  and $g_{F_2}$.
   \item[$\circ$]   If $F_1$ and $F_2$ are bilipschitzly equivalent, then so are  $g_{F_1}$  and $g_{F_2}$.
\end{itemize}

\medskip

Note that beside the the Binet-Legendre metric, we could use other procedures that associate  {a scalar product (or an ellipsoid)  } to a given convex body,  
such as the one based on the John or L\"owner ellipsoid, or the constructions in  in \cite{LYZ2000,LYZ2005, MRTZ, 
Sz1, Sz2}.  { It is however not completely clear whether the above properties, in particular the smoothness or $C^0$-stability, still hold for those
alternative constructions.}

\medskip

We  shall give concrete applications of the Binet-Legendre metric to the solution of the following  seven geometric problems: 

\smallskip

\begin{enumerate}[(1)]
  \item   A generalization of the result of Wang \cite{wang}  about the possible dimensions of the isometry groups of Finsler manifolds to  manifolds of all  dimensions.
  \item The description of local conformal maps between Minkowski spaces, thus answering a
       question raised by Matsumoto in \cite{Matsumoto}.
  \item The description of  all Finsler spaces admitting a non trivial self-similarity.
  \item The topological classification of conformally flat compact Finsler manifolds.
  \item The description of all conformal self maps in a Finsler manifold.
  \item A short proof of a theorem of Szab\'o on Berwald spaces.
  \item A positive solution to the conjecture of S. Deng and Z. Hou \cite{Deng}  stating that a  locally symmetric Finsler space is Berwald. 
\end{enumerate}

{ Most of these problems are related  to  conformal or isometric mappings of  Finsler metrics. Then, the Binet-Legendge metric enjoys  the same geometric condition  as the given Finsler metric. In most problems  we consider (the exceptions are  Problem (4) and, to a certain extend,  Problem (3)), the Riemannian analog of the problem  is well-studied or can be easily described. This gives us an additional information about the geometry of the manifold that we use in the solution of the above mentioned problems. There is no general schematic way how one can use this additional Riemannian information; in certain cases it  is straightforward and in certain cases it is   tricky. } 

\smallskip

An additional result of our paper is the construction of  a family of new scalar invariants of Finsler manifolds extending the well known Minkowski functionals from convex geometry. The invariants can be effectively  calculated numerically and  one can use them to decide whether two explicitly giving  Finsler metrics are conformally equivalent or isometric.  

\medskip

The paper is organized as follows.
In section \ref{sec.partsmooth}  we introduce \emph{partially smooth} Finsler metrics and show that the corresponding  Binet-Legendre metrics  are smooth.  
The  sections \ref{sec.killing} to \ref{sec.LocSym} are devoted to the solutions of the aforementioned geometric problems.   In  section \ref{sec.MinkowskiFunctionnals}  we use the Binet-Legendre metric to produce new  conformal invariants of Finsler metrics.
In the appendix,  we  rapidly prove the basic properties of the Binet-Legendre construction.

\section{Partially smooth Finsler metrics}  \label{sec.partsmooth}

There are very natural examples of Finsler metrics which are not smooth in the usual sense. In this section we discuss a notion of smoothness that is weaker than the one usually considered in Finsler geometry
and yet is still emeanable to the techniques of differential geometry via the use of the Binet-Legendre metric.
A {\it homogeneous diffeomorphism} of a finite-dimensional vector space $V$ 
is a diffeomorphism $A:V\setminus \{0\} \to V\setminus \{0\}$ such that for every $\lambda>0 $ and  for every $v\in V$, $v\ne 0$ 
we have $A(\lambda v) = \lambda A(v)$.  

\smallskip
A  {\it field of homogeneous  diffeomorphisms} of $TM$ is a diffeomorphism $A:TM^0 \to  TM^0$, where $TM^0= TM\setminus(\textrm{the zero section})$,
such that the restriction $A_x=A_{|T_xM}$ is a  homogeneous diffeomorphism of 
the tangent space $T_xM$ for every $x\in M$.

\begin{definition}\label{def.partsmooth} 
Let $(M,F)$ be a Finsler manifold and $U\subseteq M$ the domain of some coordinate system  $(x_1,...,x_n)$.
Then $F$ is said to be  $C^k$-\emph{partially smooth in the coordinates $x_i$} if there exists  a $C^k$-smooth field of homogeneous  diffeomorphisms  $A:TU^0\to  TU^0$ such that  the function 
$x \mapsto F(x,A_x(\xi))$ is of class $C^k$ in $U$ for any fixed $\xi \in \mathbb{R}^n$.  
\end{definition}

In this definition, we use the identification $TU = U\times \mathbb{R}^n$ defined by the coordinate system. The vector field $\xi$ is thus ``constant in the coordinate system $x_i$''.  

\begin{lemma} \label{lem:part}
 Let $U$ be some domain in the  Finsler manifold $(M,F)$. If $F$ is  $C^k$-partially smooth in some coordinate system  on $U$, then it is partially smooth in any coordinate system on $U$.
\end{lemma}

Partial smoothness in some coordinate domain is thus in fact an intrinsic notion, and we are led to the following global definition.

\begin{definition}\label{def.globpartsmooth}
A Finsler manifold is  $C^k$-\emph{partially smooth} if it is $C^k$-partially smooth in some neighborhood of any of its point.
\end{definition}

\textbf{Proof of  Lemma \ref{lem:part}.}   From the local nature of the concept, on may assume that $M$ is a domain $U \subset \r^n$ and that the Finsler metric $F$ is $C^k$-partially smooth in the natural coordinates of $\r^n$. We consider a field  $A$ of homogeneous diffeomorphisms as in the the definition  \ref{def.partsmooth}: 
 $A$ is $C^k -$ smooth and  the mapping  $x \mapsto F(x,A_x(\xi))$ is of class $C^k$ in $U$ for any constant vector field $\xi$.  Let $y_j$ be another coordinate system on $U$, specifically, let $\phi : V \to U$ be a diffeomorphism from some domain $V$ onto $U$ and set $x = \phi (y)$. The Finsler structure $F$ on $U$ transforms into the
Finsler structure $\tilde{F}$ on $V$ defined as
$$
  \tilde{F}(y, \xi) = F(\phi (y), d\phi_y(\xi)).
$$
Define now the field of  homogeneous diffeomorphism $\tilde{A} $ as $\tilde{A}_y  = d\phi_y^{-1}\circ A_{\phi(y)}$. For any fixed vector $\xi \in \r^n$,
the function 
$$
 V \ni y \mapsto  \tilde{F}(y, \tilde{A}_y  (\xi)) = F(\phi (y), d\phi_y \circ  \tilde{A}_y (\xi)) = F(\phi (y), A_{\phi(y)} (\xi))
$$  
is the composition of the $C^k$ functions $\phi : V \to U$ and $x \mapsto F(x,A_x(\xi))$, therefore
$$
  y \mapsto \tilde{F}(y, \tilde{A}_y ( \xi))
$$
is of class $C^k$ for any constant vector $\xi$ and we conclude that $\tilde{F}$ is partially smooth in the coordinates $y_j$.
\qed

\medskip

Let us give some examples of  partially smooth Finsler metrics.  
\begin{enumerate}[(a)]
  \item Every smooth Finsler metric is partially smooth
  \item  \label{min} A  Minkowski  Finsler metric $F(\xi)$   on   $\mathbb{R}^n$  is partially smooth.  Indeed, we canonically identify $T\mathbb{R}^n$ with $\mathbb{R}^n\times  \mathbb{R}^n$ and look at $F$ as a ``function of $2$ variables which is constant in the first variable": $F(x,\xi) = F(\xi)$, i.e., the field $A$ of the homogeneous  diffeomorphisms consists of identities $\textrm{id}_x:T_xM\to T_xM$.  
  \item { Let  $F_1$ and $F_2$ be  Finsler metrics on the same manifold such that $F_1$ is partially 
  smooth and $F_2$ is smooth. 
   Let  $h_1, h_2 : M \to [0,\infty )$  be   smooth  nonnegative)
    functions on $M$ such that  $h_1(x) + h_2(x) > 0$ for all $x\in M$. Then, the following Finsler metric 
  $$
   F(x,\xi) = h_1(x) F_1(x,\xi) + h_2(x) F_2(x,\xi) 
  $$
  is again a  partially smooth Finsler metric.}
  
  \item As a special case of the previous example, consider 
   the Finsler metric on $M =\mathbb{R}^2$  given by
   $$
     F(x_1,x_2, \xi_1, \xi_2) = (1-f(x_1))\cdot (|\xi_1|+ |\xi_2|) + f(x_1)\cdot \sqrt{\xi_1^2+ \xi_2^2}
   $$
   where $f : \r \to [0,1]$ is a smooth function such that $f(x) = 0$ for $x\leq 0$ and $f(x) = 1$ for $x \geq 1$. The Finsler
   metric $F$ is partially smooth, it is independent of the variable $x_2$ and it interpolates
   from the $L^1$ norm on the plane to the euclidean ($L^2$) norm as $x_1$ varies from $0$ to $1$.

 \begin{figure} 
 \begin{minipage}{.33\textwidth} 
\includegraphics[width=.6\textwidth]{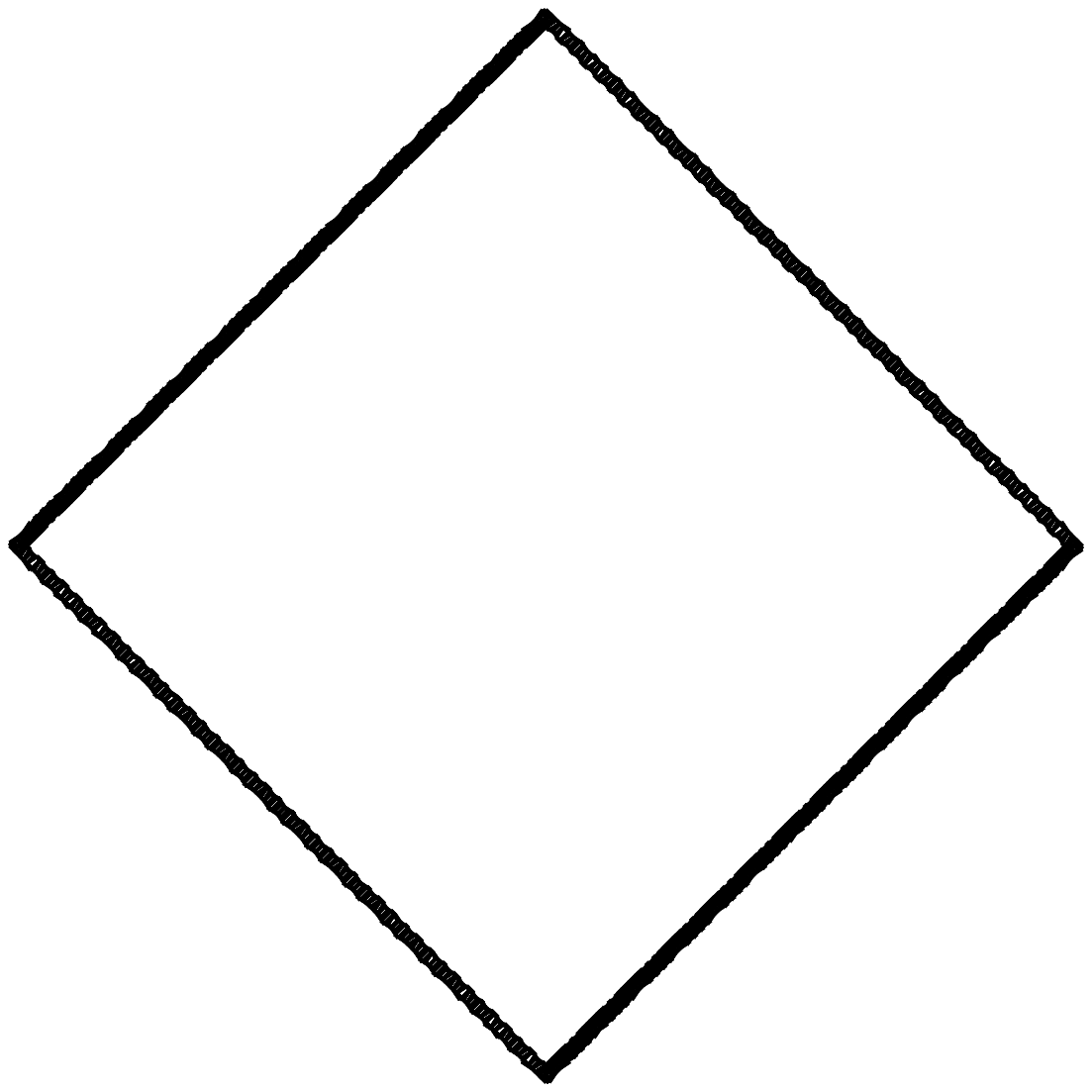}   
\centerline{\small Unit ball for   $x_1 \le -1 $} 
\end{minipage}
 \begin{minipage}{.33\textwidth} 
\includegraphics[width=.55\textwidth]{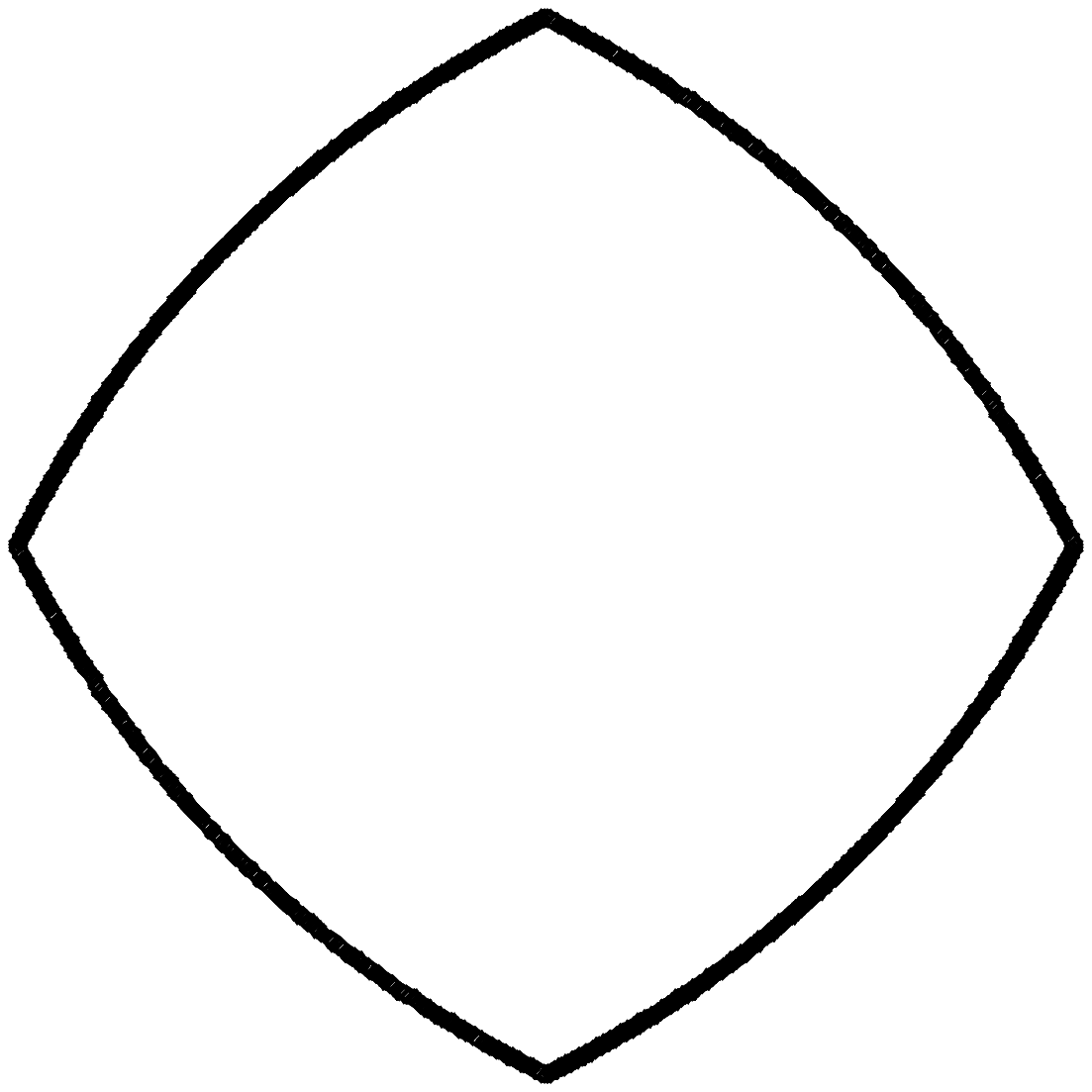}   
\centerline{\small Unit ball for   $x_1 = 0 $} 
\end{minipage}\begin{minipage}{.33\textwidth} 
\includegraphics[width=.55\textwidth]{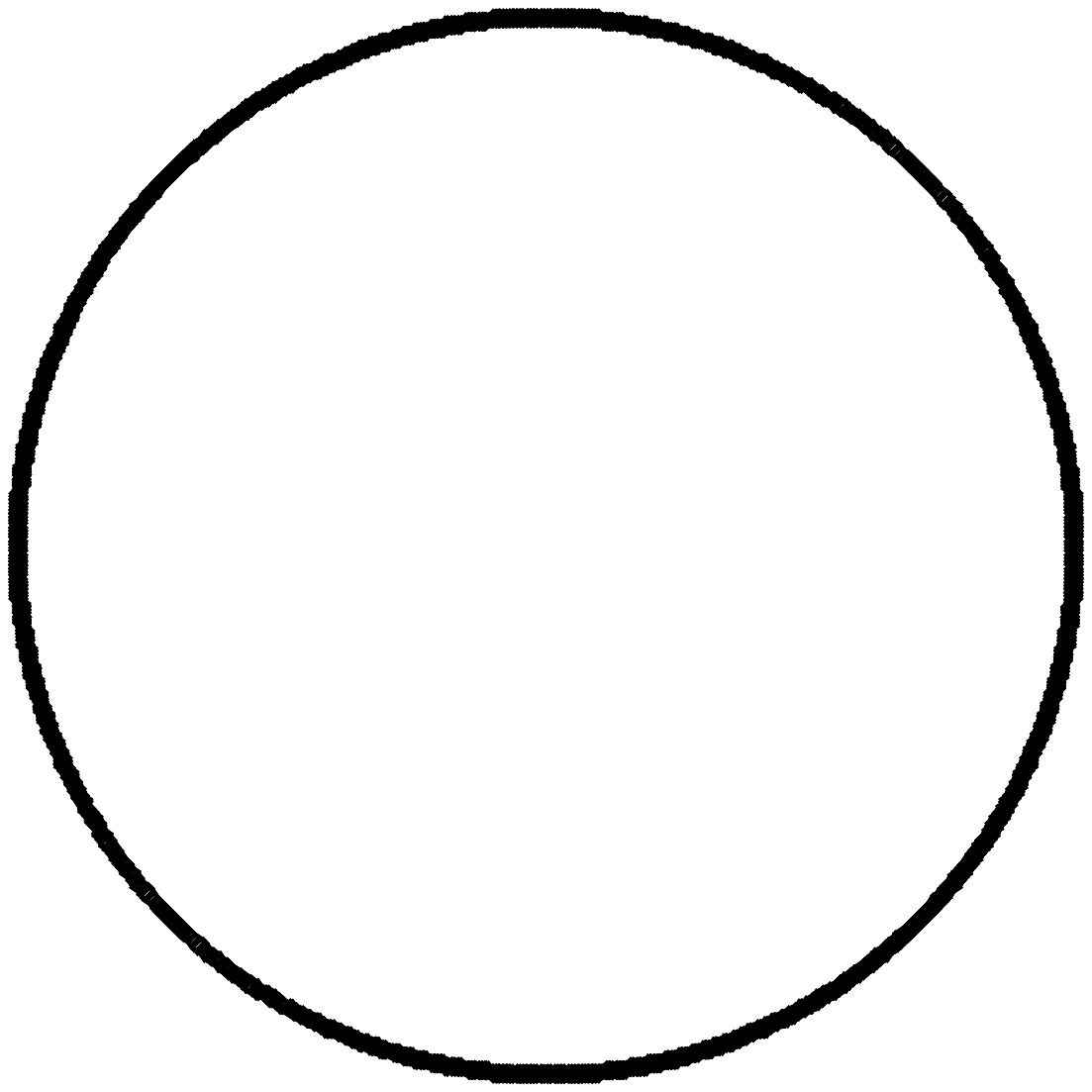}   
\centerline{\small Unit ball for   $x_1 \ge 1 $} 
\end{minipage}
\caption{The unit balls of example (d) for different values of $x_1$.}\label{pic5} \end{figure} 
\weg{\item Choose a fixed Minkowski norm $F_0$ on $\r^n$ and take a smooth matrix-valued function  $A:\mathbb{R}^n\to GL(\mathbb{R}^n)$. The  following metric $F(x,\xi)= F_0(A^{-1}(x)\cdot \xi)$   is partially smooth   and it is  smooth if and only if the Minkowski norm $F_0$ is smooth on $\r^n \setminus \{ 0 \}$.}

\item  { Let   $F$ be a   partially smooth metric on $M$.   Consider a field $A_x : T_xM \to T_xM$  of invertible endomorphisms of the tangent bundle (i.e.  $A$ is an invertible $(1,1)$ tensor field), and the  new Finsler structure defined
by
 $F_A(x,\xi) = F(x, A_x (\xi))$.
(Observe that any  Riemannian metric on a domain in $\r^n$ can be obtained
from the euclidean metric by this procedure.)  Then,  this  metrics    is partially smooth. } 
 
\item  \label{33} Consider smooth  functions $f_1,...,f_n:\mathbb{R}^2\to \mathbb{R}^2$  such that for every $x\in \mathbb{R}^2 $
the points $f_1(x), ..., f_n(x)$ are the vertices of a convex polygon $P_x$  such that the point $0$ lies in its interior. We identify $T\mathbb{R}^2$ with $\mathbb{R}^2\times \mathbb{R}^2$ and consider the Finsler metric whose $\Omega_x=P_x$  at every $x\in\mathbb{R}^2$. Then, this metric is partially smooth. 
\end{enumerate}
  
The latter example was in fact one of our original motivations for introducing the notion of partially smooth Finsler metric.
This example also suggests the following remark:  Finsler geometry can be used to describe 
certain phenomena in natural sciences  (such as light prolongation in crystals  or certain diffusion processes in organic  cells), but to
use Finsler geometry in such context, one needs to accept non-smooth metrics and the class of partially smooth Finsler metrics seems
quite appropriate.   Indeed, the cells or crystals  can be viewed as   a field of convex bodies at every point of $\r^3$ or of $\r^2$ and { can be described  by a  Finsler metric}. In particular the Finsler metric in example \eqref{33} above  could  be relevant in  describing crystal structures.

 \medskip
 
The notion of partially smooth Finsler metrics is mainly motivated by the following result.

 \begin{theorem}\label{th.smooth}
 The Binet-Legendre metric of a $C^k$-partially smooth  Finsler manifold is a Riemannian metric of class $C^k$
 on that manifold.
\end{theorem}
 
 \smallskip

\textbf{Proof.}  Let $U \subseteq M$ be the domain of some coordinate system $x_1,\dots, x_n$. 
We first prove that the function  $x \mapsto \Vol (\Omega_x)$ is of class $C^k$ in $U$ where 
$\Omega_x \subseteq T_xU = \r^n$ the Finsler unit ball and $ \Vol (\Omega_x)$ is its euclidean volume.

By hypothesis, there exists a $C^k$ field of homogeneous diffeomorphisms  $A:TU^0\to  TU^0$ such that   
$x \mapsto F(x,A_x(\xi))$ is of class $C^k$  for any fixed $\xi \in \mathbb{R}^n$.
Let us define   $\Omega'_x = A_x^{-1}(\Omega_x)$. Writing $\xi' = A_x(\xi)$, we have
$$
 \Omega'_x = \{\xi' \in \r^n \tq F(x,A_x(\xi')) < 1\},
$$
and
$$
\Vol(\Omega_x) =  \int_{\Omega_x}d\xi  =\int_{F(x,A_x(\xi'))<1} \jac(A_x)(\xi') d\xi'.
$$
Using  polar coordinates $\xi' = r\cdot u$, with $u\in S^{n-1}$, this gives
$$
\Vol(\Omega_x) 
 = \int_{S^{n-1}} \left(\int_{r = 0}^{1/F(x, A(u))}  \, \jac(A_x)(r\cdot u)  r^{n-1} \, dr \right)du, 
$$
where $du$ stands for the spherical measure on $S^{n-1}$ and $\jac(A_x)$ is the Jacobian determinant $\det\left(\frac{\partial \xi }{\partial \xi'}\right)$. 
Since the functions $\jac(A_x)$  and the  bound  ${1/F(x, A(u))}$
 $C^k-$ smoothly depend on $x$,   the integral 
$$
I(x,u)= \int_{r = 0}^{1/F(x, A(u))} \jac(A_x)(r\cdot u) r^{n-1} \, dr 
$$ 
also smoothly depends on $x$. Then,  
$$
  \Vol (\Omega_x)= \int_{S^{n-1}} I(x,u)
$$  
smoothly depends on $x$ as we claimed. 

 \smallskip
 
The proof for the Binet-Legendre metric is similar. 
  It suffices to prove that the dual metric $g^*F$ is smooth, i.e.
that $x\mapsto (g_{F})_x^*(\theta,\theta)$ is smooth in $U$ for any fixed covector $\theta : \r^n \to \r$. 
We denote by $\Theta(\xi) $ the function $ \theta(\xi)^2$ and by $\widetilde\Theta$ the function $\Theta\circ A_x$.  
Arguing as above and using formula (\ref{eq.defBLdual}), we have
\begin{eqnarray*}
\tfrac{\Vol (\Omega_x) }{(n+2)} \cdot  {g_{F}}_x^*(\theta,\theta) &=&  \int_{\Omega_x} \Theta (\xi) d\xi 
=
 \int_{F(x,A_x(\xi'))<1} \Theta(A_x(\xi)) \jac(A_x)(\xi') d\xi'  
\\
&=&   \int_{S^{n-1}} \left(\int_{r = 0}^{1/F(x, A(u))} \widetilde\Theta(r\cdot u) \, \jac(A_x)(r\cdot u)  r^{n-1} \, dr \right)du  
\end{eqnarray*}
This is again a $C^k$ function of $x\in U$, which completes the proof.
\qed

\section{On the number of Killing vector fields} \label{sec.killing}

By Theorem \ref{th.BLproperties1}(\ref{(d)}), the  group of isometries of a  partially smooth Finsler manifold $(M,F)$ is a subgroup of the group of isometries of $(M,g_F)$. It is  a closed subgroup and therefore it is a Lie group and its dimension is at most $\frac{1}{2} n (n+1)$; for smooth strongly convex Finsler metrics this statement is known, see \cite[Theorem 3.3]{DH1}. 

In 1947,  H.C. Wang proved that a smooth and strongly convex $n$-dimensional Finsler manifold of dimension $n \ne 2, 4$ is Riemannian if its group of isometries has dimension  greater than $\frac{n(n-1)}{2} +1$, see \cite{wang,Yano}.  Our next result extends Wang's theorem to all dimensions.  Our proof is more direct and also works for partially smooth metrics and without the strong convexity condition. This theorem gives a  positive answer to  a  question raised by  S. Deng and Z. Hou in  \cite[page 660]{Deng2}. 

\smallskip

A  vector field $K$ on a Finsler manifold $(M,F)$ is said to be a \emph{Killing vector field} if it generates a local flow $\phi^k_t$ of local isometries for the metric $F$.

\begin{theorem} \label{wang1}
Let $(M^n,F)$ be a partially  $C^2$-smooth connected Finlser manifold.
If the dimension of the space of Killing vector fields of $(M,F)$ is greater than $\frac{n(n-1)}{2} +1$, then  $F$ is actually a Riemannian metric.
\end{theorem}

\smallskip
 
Observe that the bound given in the Theorem is sharp: The (non Riemannian) Minkowski  space $\r^n$ with smooth and  strongly convex  norm
\begin{equation}\label{ }
   F(\xi) =  \left(\left( \sum_{i=1}^{n}\xi_i^2 \right)^2 + \xi_n^4\right)^{1/4}
\end{equation}
has $r=n+\dim{SO(n-1)} = \frac{n(n-1)}{2} +1$ linearly independent complete Killing vector fields.
 
 \medskip
 
\begin{proof} 
Let $r>\frac{n(n-1)}{2} +1$ be the dimension of the space of Killing vector fields. Take a point $x$ and choose $r-n$ linearly independent Killing vector  fields $K_1,\dots,K_{r- n}$ vanishing at $x$, this is possible because the dimension $T_xM$ is $n$. 
The point $x$ is then a fixed point of the corresponding local flows 
$\phi^{K_1}_t$,  ... , $\phi^{K_{r-n}}_t$. It is obvious that any Killing vector field for $F$ is also a Killing vector field of $g_F$. In particular, for every fixed $t$,    the differentials of $\phi^{K_1}_t$,  ... , $\phi^{K_{r-n}}_t$ at $x$ are  linear isometries of $(T_xM, g_F)$.    
Let us denote by $\Phi_i \in \End(T_xM)$ the differentials  $\Phi_i = \left(\frac{d}{dt}d_x\phi_t^{K_i}\right)_{|t=0}$. We claim that 
$\Phi_1,\dots,\Phi_{r- n}$ are linearly independent. Indeed, assume that $\sum_{i=1}^{r-n}a_i\Phi_i = 0$ for some constants $a_i \in \r$ and consider the Killing field $K = \sum_{i=1}^{r-n}a_iK_i$. Let us denote by $\phi_t^K$ the (local) flow generated by $K$; because
 $\phi^{K}_t \circ \exp_x = \exp_x \circ d_x\phi^{K}_t$, we have for $y= \exp_x(\xi)$:
 $$
   K_y =\left. \frac{d}{dt}\right|_{t=0} \phi^k_t(y) = \left. \frac{d}{dt}\right|_{t=0} \exp_x(d\phi^K_t(\xi))
   = 0
 $$
since $\left(\frac{d}{dt}d_x\phi_t^{K}\right)_{|t=0}= \sum_{i=1}^{r-n}a_i\Phi_i = 0$. It follows that $K=0$ in an open neighborhood of the point $x$ implying  $K\equiv 0$ on the whole manifold. Because  $K_i$ are assumed to be linearly independent, we have $a_i = 0$ for all $i$ and $\Phi_i$ are thus linearly independent as claimed.

\smallskip

We now denote by $G\subset  SO(T_xM, g_F)$ the  smallest closed subgroup of $SO(T_xM, g_F)$  generated by  the differentials of $\phi^{K_1}_t$,  ... , $\phi^{K_{r-n}}_t$ at $x$. Its Lie algebra contains the linearly independent elements $\Phi_1, \dots, \Phi_{r-n}$ and we thus have $\dim(G) \geq r-n$.
It is known that  for every $n\ge 2$, any $r-n$-dimensional    subgroup of  the orthogonal group   $SO(n) \cong SO(T_xM, g_F)$ acts transitively on the $g_F$ unit sphere $S^{n-1}\subset T_xM$ provided $r >\frac{1}{2}n (n-1) +1$.  Indeed, for $n\ne  4$, this immediately follows for the classical result of    Montgomery and  Samelson   \cite{montgomery}: they proved that for $n\ne 4$,   there exists no proper subgroup of  $SO(n)$ of dimension greater than $\frac{(n-1)(n - 2)}{2}$.     In dimension 4, the transitivity follows for example from    \cite[\S 1]{ishihara}, where all Lie subgroups of $SO(4)$ are described.   
   
Since the action of $G$ on $T_xM$   preserves  $F$ and    $g_F$ and $G$ acts transitively on the $g_F$-sphere $S^{n1}\subset T_xM$, the ratio $F(\xi)^2/g(\xi, \xi)$ is constant for all $\xi\in T_xM^0$ implying that    $F(\xi) = \lambda(x) \cdot \sqrt{g_F(\xi, \xi)}$ for some function $\lambda : M \to \r_+$ and for all $\xi\in TM$. This proves that $F$  is  Riemannian,  furthermore, by Theorem \ref{th.BLproperties1}(b),  the coefficient $\lambda\equiv 1$ so that $g_F$ coincides with $F$ in the  sense   $g_F(\xi, \xi)=F^2(\xi)$ for all $\xi\in TM$. 
\end{proof} 

\medskip

Observe that hypothesis of $C^2$ partial smoothness of the metric was not really used in the proof, we only used  that the flows  of the Killing vector fields  are of class $C^1$, which  is  automatically fulfilled if the metric is $C^2$-partially smooth.

\medskip

\begin{remark} 
Smooth Riemannian manifolds with large groups of isometries have been studied thoroughly, see e.g.  \cite{Kobayashi72} for a  survey of classical results. In particular, connected Riemannian manifolds with more than  $\frac{1}{2}n(n-1) +1$  Killing vector fields are classified as follows.   Let $r$ be the dimension of the space of Killing vector fields. Then
 \begin{enumerate}
  \item If  $r > \frac{n(n-1)}{2} +1$  and $n\neq 4$, then  $g$ has constant sectional curvature, see  \cite{Yano}. 
  \item If $n=4$ and $r> \frac{1}{2}n(n-1) +2=8 $, then $g$ also 
   has constant sectional curvature, see \cite{ishihara}.
  \item If $n=4$ and   $r>\frac{1}{2}n(n-1) +1=7$ then either  $M$
  is K\"ahlerian with constant holomorphic sectional curvature (in this case, $r=8$), or  
  $M$ has constant sectional curvature,  see   \cite[Theorem A']{ishihara}.
\end{enumerate}
Note that although the cited references assume the Killing vector fields to be complete, the proofs work without this hypothesis;
 a Riemannian manifold with constant sectional curvature locally has $\tfrac{n(n+1)}{2}$ linearly independent Killing fields.
\end{remark}

\section{The Liouville Theorem for Minkowski spaces and the solution to a problem by  Matsumoto}
\label{sec.matsumoto}

One of the most famous theorem of Joseph Liouville states that any conformal transformation of a domain
in $\r^3$ to another such domain is either the restriction of a similarity or the composition of an isometry  with an inversion,
it is, in other words, the restriction of a M\"obius transformation.
This result has been announced in 1850 in \cite{LiouvilleA}, and the proof appeared as a note in the fifth edition of 
Monge's book \emph{Application de l'analyse à la géométrie}  \cite{LiouvilleB}. It is well known that this Theorem also holds
in $\r^n$ for $n \geq 3$. By contrast, in dimension $2$ the Cauchy-Riemann equations  imply  that a transformation is
conformal if and only if it is either holomorphic or antiholomorphic.

\smallskip

Our next statement says that Liouville's Theorem still holds in non euclidean  Minkowski spaces.
We have in fact a stronger result.

\begin{theorem}\label{th.MLiouville}
 Let $(V_1,F_1)$ and $(V_2,F_2)$ be two non-euclidean  Minkowski spaces of the same dimension $n \geq 2$. If $f : U_1 \to U_2$ is
 a conformal map between two domains  $U_1\subset V_1$  and $U_2\subset V_2$, then  $(V_1,F_1)$ and $(V_2,F_2)$ 
 are isometric and $f$ is (the restriction of) a similarity, that is the composition of an isometry and a homothety $x\mapsto \const \cdot x$.
\end{theorem}

\begin{Remark} In the last sentence of the paper  \cite{Matsumoto}, M. Matsumoto   asked whether 
there exist  two locally Minkowski spaces which are conformal
to each other. The above theorem shows that the answer to this question is negative 
unless the metrics are Euclidean or the conformal correspondence is a similarity.
\end{Remark}

\begin{proof} We will first prove the theorem for $n\ge 3$. 
 Fix a point $x \in U_1 \subseteq V_1$ and let $y = f(x)\in U_2 \subseteq  V_2$ be the image point. Because $f$ is a conformal map,
 we have $df_x^* (F_2) = \lambda(x) F_1$ for some function $\lambda (x) >0$, hence the map $\frac{1}{\lambda(x)}\cdot df_x$ is an isometry from
 $(T_xV_1,F_1)$ to $(T_yV_2,F_2)$, but since a Minkowski space is isometric to its tangent space at any point
 it follows that  $(V_1,F_1)$ and $(V_2,F_2)$  are isometric.
 
  \smallskip
  
 From now on, we assume that $V_1 = V_2 = \r^n$ and  $F_1 = F_2 = F$ is an arbitrary non euclidean Minkowski norm.
 Changing coordinates if necessary, one may also assume that the Binet-Legendre scalar product $g_F$ of $F$ 
 is the standard scalar product $\langle \ , \ \rangle$ of $\r^n$. It follows that 
 $f$ is a conformal map in the usual sense between two domains $U,V \subseteq \r^n$.
 
 \smallskip

 By the classical  Liouville   Theorem,   $f$   is the restriction of a M\"obius transformation, and such a  map is  known to be either
 a similarity or the composition of  an isometry  and an inversion.  We thus only need to prove that   
 \emph{ the composition of  an isometry and an inversion cannot be a  conformal map of some non euclidean  Minkowski norm $F$ on $\r^n$. } 
 
  \smallskip
  
  {
We now prove the last assertion  by contradiction. The map $f$ is of the type
 $$
     f(x) = Q\left( r^2\cdot  \frac{x-c}{|x-c|^2}\right) + b,
 $$
where $r>0$ and $Q$ is a linear  orthogonal transformation. The differential of $f$ at a point $x$ is then
$$
  df_x(\xi) = r^2 Q\left( \frac{|x-c|^2\cdot \xi -2 \langle x-c, \xi \rangle \cdot (x-c)}{|x-c|^4}\right).
$$
Observe that if $x=c+r\cdot v$ with $|v|=1$, then $df_x = Q\circ R_{v}$ where $R_v$ is the reflection across the hyperplane $v^{\bot}$. In particular $df_x$ is an isometry for the Euclidean norm. Now since $f$ is a conformal map and the Binet-Legendre scalar product coincides with the  standard scalar product on $\r^n$, Proposition \ref{prop.binet2} (c) implies that $df_x$ is also an isometry for $F$, that is $F(Q\circ R_{v} (\xi)) = F(\xi)$ for every $\xi$ and every unit vector $v$.
Since  the  mappings of the form  $\xi \mapsto Q(R_v(\xi))$, where $v\in S^{n-1}$ generates the orthogonal group, our Minkowski norm $F$ is $O(n)$-invariant and is therefore Euclidean.  The theorem is proved for $n\ge 3$. 
}

  \smallskip
  
 Let us now prove it for $n=2$.  We again consider $\r^2$ with a fixed Minkowski metric which we denote by $F$, and  assume  that  $g_F$ is the standard flat metric.  
 Let us  use the conformal structure to construct a family of parallel lines on $\r^2$.  
 Take a point $x$ and consider the unit circle   $S^1_x \subset T_x\r^2$  in the metric $g_F$. We take 
 a connected  component $I^0_{max}$ of the `maximal' set 
 $$ 
   I_{max} = \{\xi \in S_1(x) \tq   F(\xi) = \max_{\eta_\in S_x^1}F(\eta) \}.    
$$  
The set  $I^0_{max}$ cannot coincide with  the whole $S_1(x)$  and is  therefore a connected interval. Let $\xi\in  S_1(x)$ be its midpoint 
 (with respect to the metric on $S^1_x$ induced by $g_F$).

  \smallskip
 
 The vector  $\xi$ is not always unique (the set $I_{max}$ can have more that one connected components, and every connected components has its own midpoint). 
 We choose one of it.
 
  \smallskip
     
Note  that the construction of the vector $\xi$ is conformally invariant in the following sense:  
if we multiply  $F$  at a point $x$ by a number  $\lambda$,  the vector $\xi$ is divided  by $\lambda$, so the direction of this vector field remains the same.   
 
Now let us  extend the vector to all points of $\r^2$ by parallel translations, thus obtaining a vector field that we denote by $\xi$.  
Let $f:U_1\to U_2$ be a conformal (i.e. holomorphic or antiholomorphic) mapping. Then, it sends the vector field   $\xi$ to another vector field $\xi' = f_*(\xi)$ that satisfies the properties by construction:
 \begin{enumerate}
  \item $\xi'$ is a smooth vector field.
  \item At every point, $\xi'$  is the mid vector of a connected component of $I_{max}$.
\end{enumerate}
Therefore  the integral curves of $\xi'$ are parallel lines in $\mathbb{R}^2$. It is well known (and easy to check) that    a holomorphic or antiholomorphic
map that sends   a family of parallel lines to a  family of parallel lines is of the type $f(z) = az+b$ or $f(z) = a\overline{z}+b$  with $a,b\in \mathbb{C}$,
$a\neq 0$. Thus $f$ is a similarity and the proof is complete.
 \end{proof}

\section{Conformally flat compact Finsler Manifolds}
\label{sec.CFlatCompact}

A Finsler manifold  $(M,F)$ is \emph{conformally flat}, if there is an atlas whose changes of coordinates
are conformal diffeomorphisms between open sets in some Minkowski space. 
Assuming $M$ to be non Riemannian, it follows from Theorem \ref{th.MLiouville}, that these changes of coordinates are euclidean similarities. 
The manifold $M$ carries therefore a similarity structure.
It turns out that compact manifolds with a similarity structure have been topologically classified by
N. H. Kuiper  and D. Fried:  they are either Bieberbach manifolds (i.e. $\r^n/\Gamma$, where $\Gamma$ is some crystallographic group of $\r^n$), or they are Hopf-manifolds i.e. compact quotients of $\r^n\setminus \{0\} = S^{n-1}\times \r_+$  by  a group $G$
which is a semi-direct product of an infinite cyclic group with a finite subgroup of $O(n+1)$
see  \cite{Fried, Kuiper,VR}.
We thus conclude:

\begin{theorem}
A partially smooth connected compact conformally flat non Riemannian Finsler manifold is either a  Bieberbach manifolds or  
a  Hopf manifolds. In particular, it is finitely covered either by a torus $T^n$ or by $S^{n-1}\times  {S}^1$.
\end{theorem}

\smallskip

The structure of Riemannian  conformally flat manifold is more complicated, see   the discussions in  \cite{Kul,Matsumoto92,SY}.

\section{Finsler spaces with a non trivial self-similarity} \label{sec.similarity}

The next theorem concerns forward  complete Finsler manifolds. Recall that the distance $d(x,y)$ between two points $x$ and $y$ on a Finsler manifold $(M,F)$ is 
the infimum of the length 
$$L_F (\gamma ) = \int_0^1 F(\gamma (t) , \dot \gamma (t)) dt.$$
of all smooth curves  $\gamma:  [0,1] \to M $ joining these two points (i.e., $\gamma(0)=x$, $\gamma(1)=y$). 
This distance satisfies the axioms of a metric except perhaps the symmetry, i.e. the condition $d(x,y) = d(y,x)$ 
is usually not satisfied. Together with the distance comes the notion of completeness: the Finsler Manifold $(M,F)$ is said to be \emph{forward complete} if every forward Cauchy sequence converges. A sequence $\{ x_i\}\subseteq M$ is \emph{forward Cauchy} if for any $ \varepsilon >0$, there exists an integer $N$ such that $d(x_i,x_{i+k}) <  \varepsilon$ for any $i \geq N$ and $k \geq 0$. 

\medskip

A $C^1$-map  $f : (M,F) \to (M',F')$ is a \emph{similarity}  if there exists a constant $a>0$ (called the \emph{dilation constant})  such that
$F(f(x),df_x(\xi)) = a\cdot F(x,\xi)$ for all $(x,\xi) \in TM$.  It is an isometry if $a =1$. 

Clearly a similarity satisfies $d_{F'}(f(x),f(y)) = a \cdot d_F(x,y)$ for all $x,y\in M$ and it follows from the Busemann-Mayer Theorem that any  $C^1$-map satisfying this condition is a similarity in the previous sense.

\begin{theorem} \label{th.similarity}
Let $(M,F)$ be a forward complete connected $C^0$-Finsler manifold. If there exists a  non isometric self-similarity $f : M \to M$ of class $C^1$, then $(M,F)$ is a Minkowski space, that is it is isometric to any one of its tangent space.
\end{theorem}  

\medskip

The proof below is based on a blow up argument familiar in metric geometry and requires no smoothness of the Finsler metric.

\begin{proof}  
We first show that the map $f$ is a bijection. The injectivity follows from the fact that  $d(f(x),f(y))= a \cdot d(x,y)$, for any $x,y$ and $a>0$.
To show that $f$ is surjective, we observe that $f(M)\subset M$ is open since $f$ is an immersion and   $f(M)\subset M$  is closed since it is a forward complete set.
Hence $f(M) = M$ and $f$ is thus bijective.

\smallskip

Replacing  $f$ by $f^{-1}$ if necessary, one may assume that  $a<1$. We show that $f$ has a fixed point: pick 
an arbitrary point $x$ and consider the  sequence $y_k= f^k(x)$, we have then 
$$
 d(y_i, y_{i+1})= d(f^i(x), f^{i+1}(x)) = a^{i}d(x,f(x)),
$$ 
which implies that the sequence is forward Cauchy. This sequence has therefore a unique limit $x_0$ and by continuity of $f$ we have
$$
 f(x_0) = \lim_{j \to \infty} f(y_j) = \lim_{j \to \infty}  y_{j+1} = x_0,
$$
we found our fixed point $x_0$.
We now consider the Binet-Legendre Riemannian metric $g_F$, by Theorem \ref{th.BLproperties1}(\ref{(d)}),  the mapping $f$ is  a similarity also   for  $g_F$. We claim:
\begin{lemma}   \label{petr} 
Let $(M,g)$ be a $C^0$ Riemannian manifold. Assume that there exists a map $f : M \to M$ such that $d(f(x),f(y)) = a\cdot d(x,y)$ for some constant $0<a<1$
where  $d$ is  the distance function corresponding to the Riemannian metric $g$.
If $f$ has a fixed point, then $(M,g)$ is  flat, i.e., every point of $M$ has a neighborhood that is isometric to a domain in $\mathbb{R}^n$ with the standard metric.
\end{lemma}

As said before, we prove this lemma by a blow up argument\footnote{%
The proof is elementary if the metric $g$ is $C^2$: set $\kappa(x)= \max |K(\pi)|$ where $\pi$ ranges through all $2$-planes in $T_xM$ and $K$ is the sectional curvature. For a similarity $f$ with dilation constant $a$ we have $\kappa(x) = a^{2m}\kappa (f^m(x))$  thus, if $a<1$ and $\{f^m(x)\}$ converges, we have $\kappa(x) = 0$.}.
Let $x_0\in M$ be the fixed point of $f$ and choose $R$ small enough so that the closed $d$-ball  $\overline{B}_R(x_0)$ is compact. 
It suffices to show that  the restriction of the metric  $d$  to this ball is flat, since for every bounded neighborhood  $U\subseteq M$ there exists $m$ such that   $f^m(U)\subset  B_R(x_0)$.

In order to do it, we construct a sequence of flat  metrics $d_m$ on $B_R(x_0)$ such that it uniformly converges to the metric of $d$, in the sense that for every $x, y\in B_R(x_0)$  we have $d_m(x,y) {\to} d(x,y)$  uniformly as $m\to \infty$.  Choosing a smaller radius $R$ if necessary, one may assume that some coordinates $x_1, \dots , x_n$ are defined in some neighborhood of the ball $B_R(x_0)$. Assume also  that the point $x_0$ has coordinates $(0,...,0)$ and that the metric $g$ is given by the identity matrix  at the point $x_0$. In this neighborhood, we consider the flat (constant)  Riemannian metric $g_0= dx_1^2 + ...  + dx_n^2$. Both metrics $g$ and $g_0$ coincide at the point $x_0$.
The  distance in the metric $g$  is denoted by $d$ and that in the metric $g_0$ will be denoted by $d_0$. Likewise balls in the $d$-metric are denoted by $B_r(x)$ and balls in the $d_0$-metric will be denote by $B'_r(x)$. 

We take $R'$ such that $B'_{R'}(x_0)\subset B_R(x_0)$.  For every $m \in \mathbb{N}$ we define a metric $d_m$ on $B'_{R'}(x_0)$ by

$$
 d_m(x,y)= \frac{1}{a^m } d_0(f^m(x), f^m (y)).
$$ 

Let us show that the sequence of  metrics $d_m$   converges to the metric $d$. Since the metric $g$ is continuous, and since at the point $x_{0}$ the metric $g$   coincides with the metric     $g_0$, for every $\varepsilon >0$ there exists $r(\varepsilon)$ such that for every point $x\in  B'_{3r(\varepsilon)}(x_0) \cup B_{3r(\varepsilon)}(x_0)$  and for every nonzero tangent vector $\xi \in T_xM$ we have 
$$
 \frac{1}{1 + \varepsilon} \leq    \frac{\sqrt{g(\xi,\xi)}}{\sqrt{g_0(\xi, \xi)}} \leq 1 + \varepsilon.
$$ 
These inequalities immediately give the following estimates on the length of  any curve $\gamma:[0,1]\to B_{3r(\varepsilon)}(x_0)$:
$$  
   \frac{1}{1 + \varepsilon} L_g(\gamma)   \leq L_{g_0}(\gamma)  \leq (1+\varepsilon) L_g(\gamma),  
$$
Assuming  $\varepsilon<\tfrac{1}{2}$, these estimates imply that the shortest path connecting two points in 
$B_r(x_0)$ stays in the ball  $B'_{3r}(x_0)$, and symmetrically the shortest path connecting two points in 
$B'_r(x_0)$ stays in the ball  $B_{3r}(x_0)$. We therefore have the following inequalities for any $x,y \in  B'_{r(\varepsilon)}(x_0) \cap B_{r(\varepsilon)}(x_0)$:
$$
     \frac{1}{1 + \varepsilon} d(x,y)   \le d_0(x,y) \le (1+ \varepsilon) d(x,y).
$$

\smallskip  

Now take two arbitrary points $x, y\in B_R(x_0)$. 
For sufficiently large $m$,  the points  $f^m(x)$ and   $f^m(y)$ lie in $B_r(\varepsilon)(x_0)$.
 By definition,  the distance between $f^m(x)$ and $f^m(y) $ is the length of a shortest curve. 
 Since this curve lies in $B_{3r(\varepsilon)}(x_0)$, the inequalities  above imply that 
$$
      \frac{1}{1 + \varepsilon}d(f^m(x), f^m (y)) \le d_0(f^m(x), f^m(y)) \le (1+ \varepsilon) d(f^m(x), f^m (y)).
$$  

Dividing this inequality by ${a^m}$  and using the property $d(f^m(x), f^m(y))= a^m\cdot d(x,y)$   together with the definition of $d_m$ we obtain 
$$ 
      \frac{1}{1 + \varepsilon} d(x, y)\le d_m(x,y) \le  (1+ \varepsilon) d(x,y).
$$ 
Since for $x,y\in B_R(x_0)$  the function $d(x,y)$ is uniformly 
bounded by $2R$,  the metrics $d_m$ uniformly  converge to the metric $d$ as $m\to \infty$. 
Furthermore the metrics $d_m$ are clearly  flat  metrics: $B_R(x_0)$ equipped with such  metric  is  isometric  to a domain in the standard  euclidean space $\r^n$.
 
 \smallskip
 
Is it is well known that a uniform limit of flat metrics, is itself  flat.  For the sake of completeness, we give a proof
of this fact in our case.   We may assume that $R\ge 3$, otherwise we divide  the metric by a large constant.   
We will prove that the metric  $d$ in the ball $B_1(x_0)$ is  flat. 

For any $m$, we choose  an isometric embedding $\phi_m :  (\overline{B}_R(x_0), d_m) \to \r^n$
such that  $\phi_m(x_0) = 0$. Let us set  $x_j(m) = \phi_m^{-1}(e_i) \in \overline{B}_R(x_0)$
where $e_1, e_2, \dots , e_n \in \r^n$ is the standard orthonormal basis. 

 Since $\overline{B}_R(x_0)$ is compact, one can find a subsequence  $(x_1(m_i),...,x_n(m_i)) $ converging  to a tuple 
 $(x_1,...,x_n)\in B_1(x_0)\times ... \times  B_1(x_0)$. We claim that the restriction of the sequence $\phi_{m_i}$ to $B_1(x_0)$ 
 converges to a map  $\phi : B_1(x_0)\ \to \r^n$ which is an isometry.
 
 Indeed, for any $y \in \overline{B}_R(x_0)$ the point $\phi_{m_i}(y)$ is the unique point in $\r^n$ such that
  $\|\phi_{m_i}(y) \|=  d_{m_i}(x_0,y)$  and   $\|\phi_{m_i}(y) - e_j\|=  d_m(x_j,y)$ for any $j=1,...,n$. Since the
  sequence $x_j({m_i})$ converges to $x_i$ and $d_{m_i}$ converges uniformly to $d$, the sequence $\{\phi_{m_i}(y\}$
  converges to the unique point $Y\in \r^n$ such that  $\|Y \|=  d(x_0,y)$  and   $\|Y - e_j\|=  d(x_j,y)$ for any $j=1,...,n$.
  
  We denote by $\phi = \lim_{i\to \infty} \phi_{m_i}$ the limiting map. This is an isometry since
  $$
  d(y, y') =  \lim_{i\to \infty} d_{m_i}(y, y) =  \lim_{i\to \infty} \|\phi_{m_i}(y) -\phi_{m_i}(y')   \|
  =   \|\phi(y) -\phi(y')   \|.
  $$
The proof of Lemma \ref{petr} is complete.

\smallskip

The lemma just proved tells us that  a neighborhood of the point $x_0 \in M$  equipped with the metric $g_F$ is isometric to a domain in the standard euclidean space.  The  next lemma (which provides the second step  in the proof of Theorem \ref{th.similarity})  says  that the metric $F$ is isometric to a Minkowksi metric in the same neighborhood.

\begin{lemma} \label{lem.similarity}
Let $F$ be a   Finsler metric on a domain   $U\subseteq \r^n $ and let $f : U \longrightarrow U$ be a map which is a self-similarity  with dilation constant $a< 1$ for  both the Finsler metric $F$ and the standard euclidean metric $g$ on $\r^n$. If $f$ has a fixed point, then  $F$ is (the restriction of) a Minkowski metric.
\end{lemma}

Note that in the lemma we neither suppose that $F$ is complete nor that it is quasi-reversible.

\medskip

To prove this lemma, assume that $U$ contains the origin and that $0$ is the fixed point.
Then  $f$ is the restriction of a linear similarity (still denoted by  $f : \r^n \to \r^n$) and has thus the form  $f(x) = a \cdot Q(x)$, for some orthogonal transformation  $Q\in O(n)$.  By hypothesis, we have
 $$
  F (f(x), df_x(\xi)) = f^*F(x,\xi) = a  \cdot F(x,\xi)
 $$
  for any $(x,\xi) \in T\r^n = \r^n\times \r^n$.
 Because $df_x(\xi) ) =  a \cdot Q(\xi)$, we have
 $$
  F (f(x), df_x(\xi)) =      F(f(x), a \cdot Q(\xi))   =  a \cdot  F( f(x) ,Q(\xi)).
 $$
It follows from the two previous equalities that 
$$
 F(x,\xi) = F( f(x),Q(\xi)) = a^n F(f^n(x), Q^n(\xi))
 $$
  for any integer $n$.   Fix an arbitrary point $x \in \r^n$ and choose  
 a sequence $\{ n_j \} \subset \mathbb{N}$ such that $Q^{n_j} {\longrightarrow} \textrm{id}$ in $O(n)$ as $j \to \infty$,
 we then have
 $$
  F(x, \xi) =  \lim_{k \to \infty }F(f^{n_k} (x),  Q^{n_k} (\xi)) = F(0,\xi).
 $$
 This shows that $F(x, \xi)$ is independent of $x$, i.e.,  it is  a  Minkowski metric. The second lemma is proved

 \smallskip
 
We can now conclude the proof of Theorem \ref{th.similarity}. 
By Lemmas \ref{petr} and   \ref{lem.similarity} the metric $F$ is  a  Minkowski metric  in a  certain neighborhood  $U$ of $x_0$. Since  for bounded set  $U'\subset M$  there exists $m$ such that 
 $f^m(U')\subset U$,    the metric $F$ is  a Minkowski metric in some neighborhood of every point. Clearly, $M$ is simply connected. Indeed, for every loop  $\gamma$ there exists $m$ such that $f^m(\gamma)$  lies in a small neighborhood of $x_0$  and is therefore contractible. Because  $f^m$  is a homeomorphism on its image,  the loop $\gamma$ is contractible as well. 
We established that the manifold $(M,F)$ is forward complete, simply connected and locally isometric to a Minkowski space;
it is therefore globally isometric to a Minkowski  space. 
 \end{proof}

 \medskip
 
\textbf{Remark.} In the case of smooth Finsler manifolds, Theorem \ref{th.similarity} is known.
A first proof was given in \cite{HeilLaugwitz},  however R. L. Lovas, and J. Szilasi found a gap in the
argument and gave a new proof in \cite{LovasSzilasi}.

\section{Conformal transformations of (partially-smooth) Finsler metrics} \label{sec.conformal}
 
In this section, we classify all conformal transformations of an arbitrary Finsler manifold. 

\begin{definition}
A set  $S \subseteq \mbox{Diff}(M)$  of  transformations of the Finsler manifold $(M,F)$   is said to  be 
\emph{essentially conformal}  if  any $f\in S$ is a conformal transformation of $(M,F)$, but there is no conformal deformation $\lambda \cdot F$ of $F$  for which $S$ is a set of isometries.
The set $S$ of conformal transformations of $M$ is termed \emph{inessential} if it is not essentially conformal.
\end{definition}

\smallskip

\begin{theorem} \label{th.essentialmap}
Let $(M,F)$ be a connected $C^{\infty}$  partially smooth Finsler manifold, then the following conditions are equivalent.
\begin{enumerate}[a)]
  \item There exists an essentially conformal diffeomorphism $f$ of  $(M,F)$.
  \item The group of conformal diffeomorphism  of  $(M,F)$ is essential.
  \item $(M,F)$ is conformally equivalent to a Minkowski space $(\R^n, F)$ or to the canonical Riemannian sphere $(\mathbb{S}^n,g_0)$.
\end{enumerate}
\end{theorem}

\smallskip

The logic of the proof is the following: Using the Binet-Legendre construction, we  reduce this theorem to the 
Alekseevsky-Ferrand-Schoen solution to the Riemannian Lichnerowicz-Obata conjecture (see e.g. \cite{Al,Fe2,schoen}).
We then need to prove that the Finsler metric is conformally Minkowski in the non compact case and
Riemannian in the compact case. The main ideas are similar to those in  \cite{MRTZ}, but here we do not 
work with conformal vector fields.

\smallskip

\textbf{Remark.}  Note that it is  obvious that $(a) \Rightarrow (b)$, but  $(b) \Rightarrow (a)$ is not a priori a trivial fact because we could conceive of a Finsler manifold  $(M,  F)$ for which every conformal diffeomorphism would be inessential, but for which no conformal deformation $\lambda \cdot F$ of the metric would be simultaneously invariant under all conformal diffeomorphisms of $(M,  F)$.

\begin{proof}
As just observed,  $(a)$ trivially implies  $(b)$. It is also clear that   $(c) \Rightarrow (a)$, since any linear contraction of a Minkowski space and any non isometric Möbius transformation of the sphere are examples of essential conformal transformations. 
We thus only need to prove $(b) \Rightarrow (c)$.   

\smallskip

We know from  Theorem \ref{th.BLproperties1} (d)  that $f$ is also a conformal transformation for the associated Binet-Legendre metric,
 and $f : (M,g_F) \to (M,g_F)$ must be essential otherwise $f$ would be an inessential conformal
 transformation of $(M,F)$.
 
It follows that the full group of conformal transformations of $(M,g_F)$ is essential and by 
the Alekseevsky-Ferrand-Schoen Theorem, the manifold $(M,g_F)$   is either conformally equivalent to the euclidean space $\r^n$ or to the canonical Riemannian sphere $\mathbb{S}^n$. Changing the Finsler metric $F$ and correspondingly the Binet-Legendre metric $g_F$ within the same conformal class, we will assume  that  $(M,g_F)$ is in fact \emph{isometric} to $\mathbb{R}^n$ or $\mathbb{S}^n$.

\smallskip

If  $(M,g_F)$ is isometric to the euclidean space $(\r^n,g_0)$, then $f$ is a conformal transformation of $\r^n$ and it is therefore a map of the type $f(x) = a\cdot Q (X) + b$
with $Q \in O(n)$, $a > 0$ and $b \in \r^n$. Since $f$ is essential, we have $a\neq 1$ and we conclude from Lemma \ref{lem.similarity} that 
 $F$ is a Minkowski metric. Our claim is proved in this case.

\smallskip

We now assume that   $(M,g_F)$ is isometric to the canonical Riemannian sphere  
$\mathbb{S}^n$ and $f : \mathbb{S}^n \to \mathbb{S}^n$ is a
non isometric conformal map. It is well known that such a map has exactly either one or two fixed points.

\smallskip

\textbf{Case 1.}  $f$ has two fixed point.
\\  Using a stereographic projection, we identify $\mathbb{S}^n$ with $\r^n \cup \{\infty \}$ and we may assume that 
$f(\infty) = \infty$ and $f(0) = 0$. Thus $f$ induces a conformal map $f: \r^n \to \r^n$ which is of the type  $f(x) = a\cdot Q (x)$, with $Q \in O(n)$.
If $a=1$, then $f$ is an isometry of the spherical metric
$$
  g_1 = \rho^2(x) \cdot g_0, \qquad \rho(x) = 
  \frac{2}{1+ |x|^2}
$$
where $g_0 =  \sum dx_i^2$ is the standard euclidean metric. 
By hypothesis this metric $g_1$ coincides with the Binet-Legendre metric  $g_F$ of $F$, and by part (d) in Theorem \ref{th.BLproperties1}, 
the map $f$ is then also an isometry of the Finsler metric $F$, i.e. $f$ is inessential, a case that we excluded.

\smallskip

So we have  $a\neq 1$. Consider the Finsler metric $F^{+} = \rho^{-1} \cdot F$ on $\r^n$, its Binet-Legendre metric is the flat metric $g_0 = \rho^{-2} \cdot g_1$.
The map $f(x) = a\cdot Q(x)$ is a non-isometric similarity for both the Binet-Legendre metric $g_{F^{+}} = g_0$ and the  Finsler metric $F^+$ 
and we conclude  from Lemma \ref{lem.similarity} that  $F^{+}$ is  a Minkowski metric. 

\smallskip

Let $\varphi (x) = \frac{x}{|x|^2}$ be the standard inversion in $\r^n \cup \{\infty \}$. This map exchanges the two fixed points of $f$ and the previous argument shows that
$F^{-} = \rho^{-1} \cdot  \varphi^*F$ is also a Minkowski metric. Since $\varphi$ is conformal for the Binet-Legendre metrics of $F^{+}$ and $F^{-}$, the Liouville Theorem
\ref{th.MLiouville} implies that $F^{+} = \rho^{-1} \cdot F$ is an Euclidian metric $g^+$ and thus $F$ is Riemannian. Hence $F(\xi) = \sqrt{g_F(\xi,\xi)}$ is the standard metric on $\mathbb{S}^n$.

\smallskip

\textbf{Case 2.}  $f$ has exactly one fixed point.

We again identify $\mathbb{S}^n$ with $\r^n \cup \{\infty \}$ and  assume that 
$f(\infty) = \infty$. Thus $f$ induces a conformal map $f: \r^n \to \r^n$ which is of the type  $f(x) = a\cdot Q (X) + b$.
Since $f$ has no fixed point in $\r^n$, we must have $b \neq 0$ and $a=1$. Using  Lemma \ref{lem.essentialmap} below and conjugating $f$ with a translation if necessary,
we may assume that $b$ is an eigenvector of $Q$ with eigenvalue $+1$, i.e.  $Q(b) = b$.

We are thus in the following situation: our map $f$ is  $f(x) =  Q (x) + b$ where  $Q(b) = b \neq 0$ and the composition
$\tilde{f} = \varphi \circ f : \r^n \setminus \{ 0\} \to \r^n$ is conformal for the standard metric, where $\varphi$ is the inversion.
We have
$$
 \tilde{f}(x) = \frac{Q (x) + b}{|Q (x) + b|^2},
$$
and
\begin{align*}
  d \tilde{f}_x (\xi)    &=    \frac{Q (\xi)}{|Q (x) + b|^2}  -  2 \langle  Q(\xi) ,Q (x) + b    \rangle  \frac{Q (x) + b}{|Q (x) + b|^4}
  \\ &=
   \frac{1}{|f(x)|^2}  \left( Q(\xi)  - 2 \frac{\langle  Q(\xi) , f(x) \rangle}{|f(x)|^2}  \cdot  f(x) \right)
     \\ &=
   \frac{1}{|f(x)|^2} \cdot  (S_{f(x)} \circ Q)(\xi),
\end{align*}
where $ S_{f(x)}$ is the linear reflection across the hyperplane $f(x)^{\bot}$.
Since $Q(b) = b$, we have $f^n(x) = Q^n(x) + n\cdot b$, and the same calculation gives us 
$$
    d ( \tilde{f}_n )_x (\xi)  =    \frac{1}{|f^n(x)|^2} \cdot (S_{f(x)^n} \circ Q^n)(\xi),$$
for any $n \in \mathbb{N}$, 
where  $ \tilde{f}_n = \varphi \circ f^n$.
The map $ \tilde{f}_n $ is conformal for the Finsler metric $F$, we thus have 
$$F( \tilde{f}_n (x) , d (\tilde{f}_n )_x(\xi)) = \lambda_n(x) \cdot F(x, \xi)$$
for some function $\lambda_n$, therefore
\begin{align*}
 F(x, \xi) &= \frac{1}{\lambda_n(x)} \cdot F( \tilde{f}_n (x) , d (\tilde{f}_n )_x(\xi)) 
     \\ &= \mu_n(x) \cdot F( \tilde{f}_n (x) ,  (S_{f(x)^n} \circ Q^n)(\xi)),
\end{align*}
where $\mu_n(x) =  \frac{1}{|f^n(x)|^2 \lambda_n(x)}$.
Observe that $S_{f(x)^n}$ only depends on the direction of the vector $f^n(x)$, i.e.  $S_{f(x)^n}  =  S_{\frac{f(x)^n}{|f(x)^n|}}$,
and since 
$$
  \lim_{n \to \infty} \frac{f(x)^n}{|f(x)^n|}  =   \lim_{n \to \infty}  \frac{Q^n(x) + n\cdot b}{|Q^n(x) + n\cdot b|} =
 \frac{b}{|b|},
$$ 
we have
$$
    \lim_{n \to \infty} S_{f(x)^n}  =  S_b.
$$
By the compactness of the group $O(n)$, one may find a sequence $\{ n_j \} \subset \mathbb{N}$ such that $Q^{n_j} \to I$, we thus have
$$
    \lim_{j \to \infty}  \  {|f^n(x)|^2} \cdot  d ( \tilde{f}_n )_x (\xi) =   \lim_{j \to \infty}  (S_{f(x)^n} \circ Q^n)(\xi)
    = S_b (\xi).
$$
Now $\mu_n(x)$ is a bounded sequence and we may choose the subsequence $\{ n_j \}$ such that $\mu_{n_j}(x)$ converges to some number $\mu(x)$.
The previous considerations imply that
$$
  F(x, \xi) =   \lim_{j \to \infty}    \mu_{n_j}(x)\cdot F( \tilde{f}_{n_j} (x) ,  (S_{f(x)^n} \circ Q^{n_j})(\xi))
  =   \mu (x)\cdot F(0 ,S_b(\xi))
$$
for any $(x, \xi)$. It follows that $\frac{1}{\mu} F$ is a Minkowski metric.

Since the inversion $\varphi$ is conformal for the Minkowski metric $\frac{1}{\mu} F$, the Liouville Theorem \ref{th.MLiouville} implies that $F$ is in fact a Riemannian metric and thus   $F(\xi) = \sqrt{g_F(\xi,\xi)}$  is the standard metric on $\mathbb{S}^n$.
\end{proof}

\begin{lemma} \label{lem.essentialmap}
 Suppose that $f(x) = Q(x) + b$ is a fixed point free transformation of $\r^n$ with $Q \in O(n)$, then $f$ can be decomposed as 
 $$
  f = T\circ f_1\circ T^{-1}, 
 $$
 where $T$ is a translation and $f_1(x) =  Q(x) + b_1$ for some non zero vector $b_1$ such that $Q(b_1) = b_1$.
\end{lemma}

\begin{proof} 
Let us denote by $E = \{ v \in \r^n \tq Q(v) = v\}$. The decomposition $\r^n = E \oplus E^{\bot}$ is $Q$-invariant and we write $b=b_1+b_2$ with $b_1 \in E$ and 
$b_2 \in E^{\bot}$. The transformation $f_2 (x) = Q(x) + b_2$ has a fixed point $v_0\in E^{\bot}$. Indeed, $1$ is not an eigenvalue of the restriction $\left. Q \right|_{ E^{\bot}}$,
therefore the equation 
$$
 (Q-I)(v) = - b_2, \quad v\in E^{\bot}
$$
has a solution $v_0$, and we have $Q(v_0) + b_2 = v_0$.  Let us denote by $T$ the translation $T(x) = x+v_0$, we then have
\begin{align*}
  (T^{-1} \circ f \circ T) (x)  &=  (Q(x+v_0) + b)  - v_0 
    \\ & = Q(x) + (Q-I)(v_0) + (b_1+b_2) 
    \\ & = Q(x) + b_1.
\end{align*}
It is clear that $Q(b_1) = b_1$ since  $b_1 \in E$,  and $b_1\neq 0$, otherwise $f(x) =  Q(x) + b_2$ would have a fixed point. 
\end{proof}

\section{On Berwald spaces}
\label{sec.Berwald}

A $C^k$ -Berwald space is a Finsler manifold $(M,F)$ which admits a torsion free linear connection $\nabla$ which is compatible with the Finsler metric.  
More precisely, one says that a linear connection $\nabla$ on a smooth manifold is of class $C^k$ if its  Christoffel symbols in any coordinate system are of class $C^k$.  Recall that the parallel transport  associated to a $C^1$-path $\gamma:[0,1]\to M$ from $x= \gamma(0)$ to $y= \gamma(1)$ is the linear map 
 $P_{\gamma} : T_xM \to T_yM$  defined as $P_{\gamma}(\xi_t) = \xi_1\in T_yM$ where $t \to \xi_t$ is the solution to the equation $\nabla_{\dot \gamma (t)} \xi_t = 0$ such that $\xi_0 = \xi \in T_xM$. Observe that, since this ordinary differential equation is linear in $\xi_t$, there is a unique solution for any $t\in [0,1]$ even when the connection $\nabla$ is only of class $C^0$ (see \cite{Hartman}).

\medskip

\begin{definition}\label{def.berwald}
A Finsler metric $F$ on a manifold $M$ is  said to be a $C^k$-\emph{Berwald} metric  if there exists a $C^k$-smooth torsion free linear  connection $\nabla$ (called an  \emph{associated connection}) on $M$ whose   associated parallel transport  preserves the Lagrangian $F$. That is, if $\gamma:[0,1]\to M$ is a smooth path connecting the point $x= \gamma(0)$ to $y= \gamma(1)$  and $P_{\gamma} : T_xM \to T_yM$  is the associated $\nabla$-parallel transport, then   
$$
  F(y, P_{\gamma}(\xi)) = F(x, \xi)
$$
for any $\xi \in T_xM$. 
\end{definition}

\medskip

Observe that  if  an associated connection $\nabla$  of a Berwald metric $F$ is of class $C^k$,  then the  metric $F$  is $C^k$-partially smooth.

\medskip

Note that the definition given here  differs (and is more general) from that given in \cite{BCS}, but both definitions are  equivalent for $C^2$ and strongly convex Finsler metrics, see \cite[Proposition 4.3.3]{ChernShen}. 

\medskip

In 1981, Z.I. Szabò proved that for a smooth and strongly convex Berwald metric, there exists an associated connection which is the Levi-Civita of some Riemannian metric on $M$.  Later, other proofs  that do not require strict convexity were given in \cite{Ma2,Vi}. Our next result, whose proof is very simple, extends Szabò's theorem to the case of merely continuous Finsler metric.

\smallskip

\begin{theorem} \label{szabotheorem}
Let   $(M,F)$ be a $C^0$-Berwald   Finsler manifold. If $\nabla$ is an associated connection, then the parallel transport associated  to the  connection $\nabla$ preserves  the Binet-Legendre metric $g_F$.
\end{theorem}

\begin{proof} 
For any smooth path $\gamma : [0,1] \to M$, the parallel transport  $P_{\gamma}: T_xM \to T_yM$ is a linear map 
that sends the unit ball of $F$ at $x = \gamma (0)$ to the unit ball of $F$ at $y= \gamma (1)$. By  Proposition \ref{prop.binet2}(b),   the parallel transport  preserves  the Binet-Legendre metric $g_F$ as we claim. 
\end{proof}


\begin{Remark} 
 \textbf{(A)  }   The theorem implies the following extension of Szabò's theorem: \emph{Any partially $C^1$-Berwald metric has a unique associated linear connection $\nabla$ and this connection is the Levi-Civita connection of the  Binet-Legendre metric $g_{F}$. }

 \textbf{(B)  } One may in fact redefine a partially smooth Berwald metric as a Finsler metric for which the 
Levi-Civita connection of the Binet-Legendre metric preserves $F$.

\textbf{(C)  } Observe that a Finsler manifold $(M,F)$ is flat (i.e. locally Minkowski) if and only if it is Berwald and $g_F$ is a flat Riemannian metric.

\textbf{(D)}   It is now easy to produce examples of non  Berwald metrics for which all tangent spaces $T_xM$ are isometric as Minkowski spaces (such  Finsler metrics are 
called \emph{monochromatic} in \cite[\S 3.3]{Ba}).
Take a non euclidean Minkowski metric $F_0$ on $\r^n$ and let $A$ be a smooth field of endomorphisms  such that for every point $x$ the endomorphism 
$A_x$ is an orthogonal transformation for the Binet-Legendre metric: $A_x \in O(\r^n, g_F)$. Let $\tilde{F}(x,\xi) = F_0(A_x(\xi))$, by construction 
$F$ and $\tilde{F}$ have the same Binet-Legendre metric. In particular $g_{\tilde{F}}$ is flat, and all tangent spaces are isometric to $F_0$, but  
$\tilde{F}$ is Berwald if and only if $A$ is constant.

 \textbf{(E)  }  One can describe all partially smooth Berwald spaces by the following construction.
 Choose an arbitrary smooth Riemannian metric $g$ on $M$  and choose an arbitrary Minkowski norm  
 in the tangent space at some fixed point $q$ that is invariant  with respect to the holonomy
group  of $g$. Now  extend this norm to all
other tangent spaces by parallel translation with respect to the Levi-Civita
connection of  $g$. Since  the norm  is invariant with respect
to the holonomy group, the extension does not depend on the choice of
the curve connecting an arbitrary point to $q$, and is a partially smooth Berwald  Finsler  metric. 
\weg{
 \textbf{(F)  }  The previous Theorem also implies that on a Berwald space the Riemannian volume is  proportional to the Busemann-Hausdorff measure. This fact was first observed by Centore in the
 smooth and strongly convex case in  \cite[Theorem 4.6]{Centore}.  }

\end{Remark}

We see that if the holonomy group of $g_F$ acts transitively on the unit sphere in some tangent space, then the Finsler metric $F$ is actually Riemannian. When the holonomy group is not transitive, we have the following result.

\medskip
 
\begin{proposition}
Let $F$ be a $C^2$ partially smooth   nonriemannian  Berwald metric on a connected manifold $M$. Then, either  there exists  another Riemannian metric  $h$ which is affinely equivalent to $g_F$ but not proportional to $g_F$, or the metric $(M,g_F)$ is symmetric of rank $\ge 2$, or  both.
\end{proposition}

Recall that a Riemannian symmetrics space $(M,g)$ is said to be of rank $k$ if every point belongs to a subspace $E^k \subset M$ which is isometric to the euclidean space $\r^k$.

\begin{Remark} Recall that by de Rham's splitting Theorem \cite{derham},    the  existence  of  
$h$ such that it  is  not proportional to $g_F$, but is  affine equivalent to $g_F$, implies that 
 $(M,g_F)$  is   locally decomposable, in the sense that every  point of it has a neighborhood $U$ that is isometric to the direct product  of two Riemannian manifolds of positive dimensions.  If in addition  $(M, g_F)$ is complete, the universal cover  of $(M, g_F)$ is  the direct product of two complete Riemannian manifolds of positive dimensions.  
\end{Remark}

\begin{proof}  We essentially 
repeat  the argumentation of \cite{Ma2, Vi, Sz2}.  Fix a point $q\in M$. 
For  every smooth loop  $\gamma(t)$, ($0 \leq t \leq 1$)  such that  $\gamma(0)= \gamma(1)= q$, we denote by
$P_{\gamma}:T_qM\to T_qM$ the parallel transport  along  that loop with respect to the Levi-Civita connection of $g$. 
The set 
$$
 H_q= \{  P_{\gamma} \tq \textrm{$\gamma:[0,1]\to M$  smooth, $\gamma(0)=\gamma(1)=q$}\}
$$ 
is a subgroup of the group of the orthogonal transformations of $(T_qM,  g_F)$. Moreover, it is 
well known  (see for example, \cite{Ber,Si}), that  at least  one of the following conditions holds: 
\begin{enumerate} 
\item $H_q$ acts transitively on $S_1= \{ \xi \in T_qM \mid g(\xi, \xi)=1\}$,   
\item the metric $g_F$ is  symmetric of rank $\ge 2$,
\item there exists another Riemannian metric 
$h$ such that it  is  non proportional to $g_F$, but is affinely equivalent to $g_F$. 
 \end{enumerate} 
In the first case, the ratio $F(\xi) /\sqrt{g_F(\xi, \xi)}$ is a constant function on the sphere $T_qM \setminus \{0\}$ 
implying that the metric $F$ is Riemannian, which is contrary to our hypothesis. Thus either the second or the third case hold and the Proposition  is proved.
\end{proof}

\section{On locally symmetric Finsler spaces} \label{sec.LocSym}

\begin{definition}
The Finsler manifold $(M,F)$ is called \emph{locally symmetric},   if for every point $x\in M$ there exists $r=r(x)>0$ and an isometry $\tilde I_x: B_r(x)\to B_r(x)$ (called the \emph{reflection} at $x$) such that $\tilde I_x(x) = x$ and $d_x(\tilde I_x) = -\mathrm{id} : T_xM \to T_xM$. The largest $r(x)$ satisfying this condition is called the {\it symmetry radius} at $x$.
The manifold  $(M,F)$ is called \emph{globally symmetric}  if  
the reflection $\tilde I_x$ can be extended to a global isometry: $\tilde I_x: M \to M$.
\end{definition}
 
 \begin{theorem} \label{locsym}   
 Let $(M,F)$ be a $C^2$-smooth Finsler manifold. If $(M,F)$ is locally symmetric, then $F$ is 
$C^{\infty}$-Berwald \footnote{According to Definition \ref{def.berwald}, it means that the associated connection $\nabla$ is $C^\infty$-smooth, we can 
in fact prove that the associated connection  is $C^\omega$;  but this does not imply that the metric $F$ itself is $C^\infty$.}.  
 \end{theorem} 
 
\begin{Remark}  
This theorem answers positively a conjecture stated in \cite{Deng}, where  it  
has been proved for globally symmetric spaces,  see also \cite[\S 49]{BusGeod} and  \cite{Foulon,Kim}. 
\end{Remark} 
 

\medskip
 
\begin{proof} We will first prove the Theorem under the additional assumption that the metric $F$  is strongly convex. 
Since every local isometry for the Finsler metric $F$ is also an isometry for the Binet-Legendre metric  $g_F$, it follows that $(M,g_F)$ is a Riemannian locally symmetric space. 

In what follows, it will be convenient to use  tilde-notation for the ``Finsler'' objects, and the untilde notation for the analogous objects for the  Binet-Legendre metric  $g_F$ (for example $B_r(x)$ will denote the $r-$ball in $g_F$, and $\tilde B_r(x)$ the $r-$ball in $F$; $\gamma(t)$ will denote $g_F$-geodesic and $\tilde \gamma(t)$ will denote $F$-geodesics).   Note that a locally symmetric space is evidently reversible, so that the distance function in $F$ is symmetric, and if $t \mapsto \tilde\gamma(t)$ is a geodesic parametrized by arclength, the reversed curve
 $t \mapsto \tilde\gamma(-t)$ is also a geodesic parametrized by arclength.
 
It is known that  a locally symmetric Riemannian manifold is locally isometric to a globally symmetric space  \cite[theorem 5.1]{Helgason} and is therefore real analytic.
Then,  for sufficiently small neighborhood  $W\subset M$ and for every $x\in W$,  the   $g_F$-reflection $I_x$  is defined globally on $W$. 

For every $x\in W$, there is also the reflection $\tilde I_x : \tilde{B}_{\tilde{r}(x)}(x) \to \tilde{B}_{\tilde{r}(x)}(x)$ for the Finsler metric. By Theorem \ref{th.BLproperties1}(\ref{(d)}), the Finsler reflection $\tilde I_x$ coincides with the restriction of the Riemannian reflection $I_x$ on   
$\tilde{B}_{\tilde{r}(x)}(x)\cap W$ .  We do not know\footnote{In \cite{Bere,BuPha} a different  definition of locally symmetric Finsler manifolds was given: it was explicitly assumed that the radius of symmetry $\tilde r(x)$ is locally bounded below.  Under this assumption, $I_x$  coincides with $\tilde I_x$ in the whole $\tilde B_{r}$, where $r$ can be universally chosen for all points $x$ of a sufficiently small neighborhood $W$. From Corollary \ref {cor.globsym} it follows, that every locally symmetric space in our definition  is also a locally symmetric in the definition of \cite{Bere,BuPha}} %
a priori whether $I_x$ is an $F-$isometry in the whole ball $B_{\rho}(x)$. 

 \medskip

\begin{claim} 
For every sufficiently small $g_F$-geodesically convex  open set  $W\subset M$ and  for every      $F$-geodesic $\tilde \gamma(t):[-\tilde\varepsilon, \tilde\varepsilon]\to W$  parameterized by arclength,
 we have  $I_{\tilde \gamma(0)} (\tilde\gamma(t))= \tilde\gamma(-t)$ for all $t\in [-\tilde\varepsilon, \tilde\varepsilon]$. 
\end{claim}

Recall that $W$ is {\it $g_F$-geodesically-convex}, if every pair of points in $W$ can be connected by a unique minimal $g_F$-geodesic and that geodesic lies in $W$. 
To prove the Claim, we take a $F$-geodesic $\tilde\gamma:[-\tilde \varepsilon, \tilde \varepsilon]\to W$, set  $x= \tilde\gamma(0)\in W$, consider the  $g_F$-reflection $I_{x}$ and the number   
\begin{equation} \label{eq.defro}
   r_0(\tilde\gamma,x ) =  \sup\{r'\in [0, \tilde \varepsilon]\mid I_x(\tilde \gamma(t))
   = \tilde \gamma(-t) \textrm{ for all $t\in [-r',r']$}   \}.  
 \end{equation} 
Since the metric $F$  is strongly convex,  
 there is a unique $F$-geodesic with any given initial vector.  Then,  because  $I_x \equiv \tilde I_x$ in  a small neighborhood of $x$, we have $ r_0(\tilde\gamma,x )>0$.
   We want to prove that $r_0(\tilde\gamma,x ) =  \tilde \varepsilon$. Let us assume that $r_0(\tilde\gamma,x ) < \tilde \varepsilon$ and derive a  contradiction.

Indeed, set  $x_+= \tilde\gamma(r_0)$ and $x_-=  \tilde\gamma(-r_0)$ and consider  (the analytical continuation of) the $g_F$-reflections $I_{x_+}, I_x, I_{x_-}$. Consider $I_{x_-}\circ  I_x\circ  I_{x_+}$.  It  is  again a $g_F$-isometry.
  Let us show that  it coincides with $I_x$. In order to do this,   we consider the $g_F$-geodesic $\gamma(t)$  containing $ x_+=  \tilde\gamma(r_0)$ and $x_-= \tilde\gamma(-r_0)$.     Reparameterizing this geodesic affinely if necessary, we may assume without loss of generality that  $\gamma(1)= x_+$ and $\gamma(-1)= x_-$. Since the neighborhood $W$ is sufficiently small, we may assume that $\gamma$ is defined at least on $[-2,2]$.
Since $I_x(x_+)= x_-\in W$ and  $I_x(x_-)= x_+\in W$, we have that  $I_x(\gamma)$ is  a shortest  $g_F$-geodesic connecting  $x_+$ to $x_-$. By convexity of $W$, we must have $I_x(\gamma)\subset W$ and   $I_x(\gamma(t))= \gamma(-t)$. In particular $I_x(\gamma(0))= \gamma(0)$.
By uniqueness of the fixed point of $I_x$ in a geodesically convex region, it follows that  $\gamma(0)= x$.  Now,
$$
   I_{x_+}(x) = \gamma(2)\, ,  I_{x}(\gamma(2)) = \gamma(-2)\, \textrm{, and } I_{x_-}(\gamma(-2)) = \gamma(0)=x
$$ 
  
This implies \ $I_{x_-}\circ I_x \circ I_{x_+}(x)= x =I_x(x)$.
We next show that \\  $d_x \left(I_{x_-}\circ I_x \circ I_{x_+}\right)= -  \mathrm{id}$.  Choose a vector $\xi \in T_xM$ and extend it as parallel vector field 
along  the geodesic $\gamma$. Since the  reflection $I_{x_-}$ leaves $\gamma$ invariant and satisfies $d_{x_-}(\xi_{x_-})= -\xi_{x_{-}}$, and since an
isometry preserves parallel vector fields, we have $(I_{x_-})_* (\xi) = -\xi$ at every point of $\gamma$. The same holds for the reflections $I_x$ and $I_{x_+}$,
therefore $ \left(I_{x_-}\circ I_x \circ I_{x_+}\right)_*\xi =  -\xi$ (for arbitrary $\xi \in T_xM$). It follows that   $d_x \left(I_{x_-}\circ I_x \circ I_{x_+}\right)=-\mathrm{id}=  d_x I_x$  and therefore  $I_{x_-}\circ I_x \circ I_{x_+}=I_x$. 
 
Now, for  $\delta>0$ small enough, the mappings $I_{x_-}$ and $I_{x_+}$  are $F-$isometries  in the $F-$balls  $\tilde B_\delta(x_{-})$ and $\tilde B_\delta(x_{+})$, respectively, see figure \ref{3}.
\begin{figure}[ht]
 \includegraphics[width=.6\textwidth]{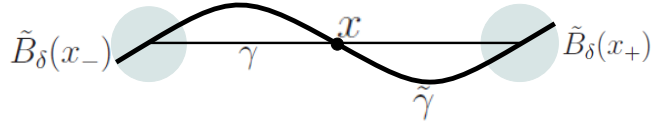}  
 \caption{The geodesics $\gamma$, $\tilde\gamma$ and the balls $\tilde B_\delta(x_{-})$,  $\tilde B_\delta(x_{+})$}\label{3}
\end{figure} 
Using again the uniqueness of an $F$-geodesic with prescribed given initial vector,
 we see that the mapping  $I_{x_-}\circ I_x \circ I_{x_+}$ sends the $F$-geodesic segment $\tilde\gamma_{|[r_0,r_0+\delta]}$ to  the $F$-geodesic segment $\tilde\gamma_{|[-r_0-\delta,-r_0]}$. 
  Replacing the isometry $I_{x_-}\circ  I_x\circ  I_{x_+}$ by the isometry $I_{x_+}\circ  I_x\circ  I_{x_-}$ in the previous argument,
 we obtain  that $I_{x_+}\circ I_x \circ I_{x_-}$ sends the $F$-geodesic segment $\tilde\gamma_{|[-r_0-\delta,-r_0]}$ to  the $F$-geodesic segment $\tilde\gamma_{|[r_0,r_0+\delta]}$.
   Since   $I_{x_-}\circ I_x \circ I_{x_+}=  I_x = I_{x_+}\circ I_x \circ I_{x_-}$, and 
 since a locally symmetric Finsler metric is  reversible,   the isometry $I_x$ has the property  
 $I_x(\tilde\gamma(t))= \tilde\gamma(-t)$ for all $t\in [-r_0-\delta,r_0+ \delta]$.  
 This gives us a  contradiction with \eqref{eq.defro} that proves the Claim.
 
\smallskip
 
Let us now show that the metrics  $g_F$ and  $F$ are affinely equivalent in the sense of \cite[p. 74]{ChernShen},
that is, for every arclength parameterised $F$-geodesic $\tilde \gamma$ there exists a nonzero constant 
$c$ such that  $\tilde \gamma(c\cdot t)$ is an  arclength parameterised $g_F$-geodesic. 
We have already seen that for a short $F$-geodesic segment, the $g_F$-geodesic segment with same endpoints  has also the same midpoint. Let us repeat the exact argument. 
Fix a sufficiently small  $g_F$-geodesically convex set  $W\subset M$ and  take a $F$-geodesic  $\tilde\gamma : [-\tilde\varepsilon, \tilde\varepsilon] \to W$. Let $\gamma : [-\varepsilon, \varepsilon] \to W$ be the unique shortest $g_F$-geodesic such that  $\gamma(-\varepsilon)= \tilde \gamma(-\tilde \varepsilon)$ and    $\gamma(\varepsilon)= \tilde \gamma(\tilde \varepsilon)$.
We assume that both geodesics are parametrised by their arclength in the metric $F$ and $g_F$ respectively.
Let $x=\tilde \gamma(0)$ be the midpoint of $\tilde \gamma$ and let $I_x$ be the $g_F$ reflection centered at $x$. Using 
the previously proved claim, we find that 
$$
  I_x(\gamma(-\varepsilon)) = I_x(\tilde \gamma(-\tilde \varepsilon))=\tilde \gamma(\tilde \varepsilon)=\gamma( \varepsilon)
$$
and likewise $I_x(\gamma(-\varepsilon)) = \gamma( \varepsilon)$.   By convexity of $W$, we must have $I_x(\gamma)\subset W$ and   $I_x(\gamma(t))= \gamma(-t)$ for all $t\in [-\varepsilon,\varepsilon]$. In particular $I_x(\gamma(0))= \gamma(0)$. By uniqueness of the fixed point of $I_x$, it follows that  $\gamma(0)= x= \tilde \gamma(0)$.
Thus, for every $F$-geodesic segment $\tilde\gamma$ in $W$, its  middle point coincides with the middle point of the unique minimal $g_F$-geodesic segment  with the same ends. 
 
 Replacing the geodesic segment $\tilde\gamma_{|{[-\tilde\varepsilon,\tilde\varepsilon]}}$ by  $\tilde\gamma_{|[{-\tilde\varepsilon,0}]}$ or by $\tilde\gamma_{|[{0,\tilde\varepsilon}]}$, we also have 
  $\gamma(-\tfrac{1}{2}\varepsilon)=\tilde \gamma(-\tfrac{1}{2}\tilde \varepsilon)$ and  $\gamma(\tfrac{1}{2}\varepsilon)=\tilde \gamma(\tfrac{1}{2}\tilde \varepsilon)$. Iterating this procedure,  
we obtain that $\gamma(s \cdot \varepsilon)  = \tilde \gamma(s \cdot \tilde \varepsilon)$ for all $s$ in a dense subset of $[-1,1]$, this  implies that the geodesic segments $\gamma$ and $\tilde  \gamma$  coincide after the  affine reparameterization $t \mapsto \tfrac{\tilde  \varepsilon}{\varepsilon}t$.
By  \cite[page 74]{ChernShen}, we obtain  that $F$ is Berwald whose  
 associated connection is the Levi-Civita connection of $g_F$. Thus, Theorem \ref{locsym} is proved for strongly convex Finsler metrics. 
 
 \smallskip

 In order to complete the proof for an arbitrary Finsler metrics $F$,   we consider the Finsler metric 
 $F_\alpha$ given by 
$$
  F_\alpha(\xi)= \sqrt{F(\xi)^2+ \alpha\cdot  g_F(\xi,\xi)},
$$
where  $\alpha>0$ is some parameter. The metric  $F_\alpha$ 
 is $C^2$-smooth and  strictly convex.  The reflections $I_x$ are evidently   isometries    of  ${F_\alpha}$, so that ${F_\alpha}$  is   locally symmetric.
We then just proved that  ${F_\alpha}$ is Berwald and its associated  connection is the Levi-Civita connection of  $g_{F_\alpha}$. Since the reflections $I_x$ are evidently  isometries    of $g_{F_\alpha}$, the metrics   $g_{F_\alpha}$ is  affinely equivalent to $g_F$ for any $\alpha>0$. Then, for every $\alpha>0$, the function $F_\alpha$ is preserved by the   parallel transport of the Levi-Civita connection of $g_F$. It follows that $F = \lim_{\alpha \to 0}F_{\alpha}$ is also preserved by the   parallel transport of the Levi-Civita connection of $g_F$ implying it is Berwald as we claimed.  
 \end{proof}

 \smallskip

\begin{corollary}\label{cor.globsym}
  Every locally symmetric $C^2$-smooth Finsler manifold is locally isometric to a  globally symmetric  Finsler space.
\end{corollary}

\begin{proof}  We consider the  Binet-Legendre metric $g_F$ of our
locally symmetric Finsler space $(M, F)$.  Since
$(M, g_F) $  is also  locally symmetric, by the classical results
of Cartan  \cite[theorem 5.1]{Helgason}, it is locally
isometric to a simply-connected globally symmetric Riemannian space $(\bar M, g)$. We
identify a small open set $U\subset M$ with an open neighborhood set $V \subset \bar M$.
This defines a Finsler metric $\bar F$ on $V$. Now extend the $\bar F$  to the whole $\bar M$ 
using the procedure in Remark (E) from section \ref{sec.Berwald}, with the help of parallel
transport of the Levi-Civita connection of $g$. Since the metric
is Berwald and the manifold is simply connected, we obtain a well
defined Finsler metric on $\bar M$. This metric (we denote it by $\bar F$) 
is evidently locally
symmetric.  Since $g$ and its isometries are real-analytic, the metric
$\bar F$ is globally symmetric as we claimed.
\end{proof}

\begin{Remark} Corollary \ref{cor.globsym}  gives us a local description of locally symmetric ($C^2$-smooth)  
Finsler spaces (in special cases this description  was obtained in  \cite{Deng,Foulon,Planche1,Planche2}). Indeed,  
take  a globally symmetric simply connected 
Riemannian  space $(M,g)$ and consider the isometry subgroup  $G$ 
generated by all reflections. The group $G$   acts transitively on $M$.  At one point $x\in M$,  
consider a smooth  Minkowski norm $F_x:T_xM\to \mathbb{R}$ such that  it is invariant with respect to the stabilizer  $G_x$ of the point $x$. Next,  extend $F_x$  to all points with the help of the action of  $G$, i.e.,  for a isometry $g\in G$ with $g(x)= y$ put $F_y(d_xg(\xi))= F_x(\xi)$. By Corollary \ref{cor.globsym}, any $C^2$-smooth locally symmetric  Finsler space is locally isometric to one constructed by this procedure. 
\end{Remark}

\section{The Minkowski functionals, and other  conformal invariants  of a Finsler manifold }
\label{sec.MinkowskiFunctionnals}

The Minkowski functionals are a family of $(n+1)$ invariants associated to a bounded convex  set $\Omega$  lying in an $n$-dimensional euclidean vector space $(E^n,g)$.
The   standard  way to define them is via the  Steiner Formula:
 \begin{equation}
    \mbox{Vol}^n(\Omega+t\mathbb{B}^n) = \sum_{j=0}^n\binom n j  \W^n_j(\Omega)t^j,
\end{equation}
where $\mathbb{B}^n \subset E^n$ is the euclidean unit ball.
Since the tangent space $T_xM$ of  a Finsler manifold $(M,F)$ is an euclidean space
(with scalar product given by the Binet-Legendre metric  $g_F$), the  Minkowski functionals  of the $F$-unit ball $\Omega_x \subset T_xM$ are well defined.
We have thus defined on the Finsler manifold  $(M,F)$ a family of $n+1$ functions:
$$
  w^n_k : M \to \r, \qquad (k =0,1,\dots , n)
$$
(the function $w^n_n$ is in fact a constant, it is the volume of the euclidean unit ball). Observe that by construction, these functions are  invariant under a conformal deformation
of the Finsler metric. It is not difficult to check that if the Finsler metric $F$  is $C^k$-partially-smooth, then the  Minkowski functionals  $w^n_k$   are $C^k$-smooth 
functions on  $M$.

\medskip

Let us construct two additional conformal invariants: At every point $x$ one sets 
$$
  \mathcal{M}(x) = \max_{0\ne \xi \in T_xM}\frac{F(x, \xi)}{\sqrt{g(\xi, \xi)}}  \quad  \textrm{  and  } 
  \quad
  \mu(x) = \min_{0\ne \xi \in T_xM}\frac{F(x, \xi)}{\sqrt{g(\xi, \xi)}}.   
$$ 
It  is easy to show that the functions  $\mathcal{M}$ and $\mu$ are continuous, but even if the 
Finsler metric is smooth, these functions may be non smooth. \weg{The Finsler  metric from in example 
 \ref{ex.nonsmooth} provides an example of such a situation. }

\smallskip

The invariants defined in the previous subsection  can be used in addressing  the following

\medskip

{\bf Equivalence problem for Finsler metrics.} {\it Let $F_1$ and $F_2$ be  Finsler metrics defined on  the discs  $U_1$ and  $U_2$: Decide if  $(U_1, F_1)$ is conformally equivalent to  $(U_2, F_2)$, in the sense that  there exists a diffeomorphism $f:U_1\to U_2$ that sends the metric $F_1$ to the metric  $\lambda \cdot F_2$ for a certain function $\lambda$ on $U_2$ ?}

\medskip

One may also consider the similar  isometric equivalence problem. This one has been addressed  by Chern in his 1948 paper \cite{Chern1948},  where he solved it by tensorial methods. His methods only works for smooth and strongly convex Finsler metrics.

\medskip

For the conformal equivalence problem, we propose the following test, which only gives a necessary condition, but which works without smoothness assumptions and is quite stable and manageable from a computational viewpoint. 
Consider the mappings $\Phi_i:U_i \to \mathbb{R}^{n+2}$ ($i=1,2$) given by
 $$
   \Phi_i(x)= (w^n_0(x),...,w^n_{n-1}(x),  \mu(x), \mathcal{M}(x)).
 $$  
If the Finsler  metrics are   conformally equivalent,  the images of these mappings (which are in general $n$-dimensional objects in $\r^{n+2}$) coincide. Thus, if  there exists at least one point that belongs to the first image and not the second,  then  the metrics are not conformally equivalent. 

\medskip

Note that the test may fail in some instances. In particular this test can never distinguish between two Riemannian metrics
and it is in fact quite delicate to decide whether two Riemannian
metrics   $g_1$ and $g_2$  are locally  isometric or conformally equivalent.

\section*{Appendix. Elementary properties of the Binet-Legendre metric} \label{sec.Binet}

\setcounter{section}{12}

Let $V$ be an $n$-dimensional real vector space and $F : V \to \r$ be a Minkowski norm on $V$. One defines a scalar product $g^*_F$
on the dual space $V^*$ by the formula
\begin{equation}\label{eq.defgF3}
   g^*_F(\theta, \varphi) =  
   \frac{(n+2)}{\lambda(\Omega)}\int_{\Omega} \left( \theta(\eta)\cdot \varphi(\eta)\right) \,  d\lambda (\eta),
\end{equation}
where $\Omega = \{\xi \in V \tq F(\xi) < 1\}$ is the unit sphere associated to $F$ and $\lambda$
 is a Lebesgue measure on $V$.
The   Binet-Legendre metric on $V$ is the scalar product $g_F$ on $V$ dual to  $g^*_F$.

\begin{proposition}\label{prop.binet2}
The  transformation  $F \mapsto g_F$ satisfies the following properties
 \begin{enumerate}[(a)]
  \item If $F$ is euclidean, i.e. $F(\xi) = \sqrt{g(\xi,\xi)}$ for some scalar product $g$, then $g_F = g$.
  \item If $A\in GL(V)$, then $g_{A^*F}=A^*g_F$.
  \item $g_{\kappa F}= \kappa^2g_F$ for any $\kappa >0$.
  \item if \ $\ds \frac{1}{c}\cdot F_1 \leq F_2 \leq c \cdot F_1$ for some constant $c>0$, then  
  $$\frac{1}{c^{2n}}\cdot g_{F_1} \leq g_{F_2} \leq c^{2n} \cdot g_{F_1}.$$
\end{enumerate}
\end{proposition}

\begin{proof}
a) Suppose $F = \sqrt{g}$ is euclidean and let $e_1,e_2, \dots , e_n$ be an orthonormal basis  on $(V,g)$
and $x_1,x_2, \dots , x_n$ be the corresponding coordinate system.  The convex
set $\Omega$ coincides with the unit ball 
$\Omega = \mathbb{B}^n = \{ x\in V \tq \sum x_i^2 < 1 \}$ and formula (\ref{eq.defgF3}) gives
$$
      g^*_F(\varepsilon_i, \varepsilon_i) =  \frac{(n+2)}{\Vol(\mathbb{B}^n)}\int_{\mathbb{B}^n} x_i^2 dx,
$$
where $\varepsilon_i = e_i ^{\flat}$.  Now the integral on the left hand side computes as follows:
$$
 \int_{\mathbb{B}^n} x_i^2 dx = \frac{1}{n} \int_{\mathbb{B}^n} \sum_{i=1}^nx_i^2 dx
 =  \frac{1}{n} \int_{S^{n-1}} \int_0^1 r^{n+1} dr d\sigma
 = \frac{\area(S^{n-1})}{n(n+2)}.
$$
But $\area(S^{n-1})=n \cdot \Vol (\mathbb{B}^n)$ and we thus have
$$
 g^*_F(\varepsilon_i, \varepsilon_i) =  \frac{(n+2)}{\Vol(\mathbb{B}^n)} \cdot  \frac{\area(S^{n-1})}{n(n+2)}
 = 1.
$$
If $j\neq i$, then 
$$
   g^*_F(\varepsilon_i, \varepsilon_j) =  \frac{(n+2)}{\Vol(\mathbb{B}^n)}\int_{\mathbb{B}^n} x_ix_j dx = 0
$$
because the function $x_ix_j$ is antisysmmetric with respect to the orthogonal transformation $x_i \mapsto - x_i$.
It follows that $\varepsilon_1, \cdots, \varepsilon_n$ is an orthonormal basis of $V^*$ for the scalar product $g^*_F$.
By duality, $e_1, \cdots, e_n$ is also an orthonormal basis of $V$ for the scalar product $g_F$ and therefore
$g_F = g$.

\smallskip

We now prove property (b). If  $A\in GL(V)$, then the unit ball $\Omega_A$ associated to $A^*F= A\circ F$ 
is the set $A^{-1}\cdot \Omega$, indeed
$$
 \Omega_A = \{ \xi \in V \tq F(A\xi) < 1 \} = \{ A^{-1}\eta \in V \tq F(\eta) < 1 \} = A^{-1}\cdot \Omega.
$$
Therefore 
\begin{align*}
  g^*_{A^*F}(\theta, \theta)  &=  \frac{(n+2)}{\lambda(A^{-1}\Omega)}\int_{A^{-1}\Omega} \theta(\eta)^2 d\lambda (\eta)
  \\ & =  \frac{(n+2)}{|\det(A^{-1})|\cdot \lambda(\Omega)}\; \int_{A^{-1}\Omega} \theta(\eta)^2  d\lambda (\eta).
\end{align*}

Setting $\xi = A\eta$, we have from the change of variable formula
$$
 \int_{A^{-1}(\Omega)} \theta(\eta)^2   d\lambda (\eta) =
 \int_{\Omega} \theta(A^{-1}\xi)^2  |\det(A^{-1})| d\lambda (\xi),
$$
and thus
$$
  g^*_{A^*F}(\theta, \theta) =  \frac{(n+2)}{\lambda(\Omega)}  \int_{\Omega} \theta(A^{-1}\xi)^2 d\lambda (\xi)
  =  g^*_{F}(\theta \circ A^{-1}, \theta \circ A^{-1}). 
$$
This is the relation between $g^*_{A^*F}$ and $g^*_{F}$. In the space $V$, we then have by duality
$$
 g_{A^*F}(\xi,\xi) =  g_{F}(A\xi,A\xi). 
$$

\smallskip

Property (c) is the special case of property (b) corresponding to scalar matrices.

\smallskip

To prove (d), let $F_1,F_2$ be two Minkowski norms satisfying $\frac{1}{c}\cdot F_1 \leq F_2 \leq c \cdot F_1$,
then the corresponding unit balls also satisfy
$$
  \frac{1}{c}\cdot \Omega_1 \subset  \Omega_2 \subset  c \cdot \Omega_1.
$$
This implies in particular that
$$
 \frac{1}{\lambda(\Omega_2)} \leq  \frac{c^n}{\lambda(\Omega_1)}.
$$
We also have
$$
 \int_{\Omega_2} \theta(\eta)^2 d\lambda (\eta) \leq \int_{c \cdot \Omega_1} \theta(\eta)^2 d\lambda (\eta)
 =  c^n \cdot \int_{\Omega_1} \theta(\xi)^2 d\lambda (\xi)
$$
(set $\xi = c\, \eta$). Therefore
$$
\frac{(n+2)}{\lambda(\Omega_2)}\int_{\Omega_2} \theta(\eta)^2 d\lambda (\eta)
   \leq  c^{2n} \cdot
   \frac{(n+2)}{\lambda(\Omega_1)}\int_{\Omega_1} \theta(\eta)^2 d\lambda (\eta),
$$
that is 
$$
   g^*_{F_2}(\theta, \theta)  
   \leq  c^{2n} \cdot  g^*_{F_1}(\theta, \theta).
$$
The dual scalar product satisfies then
$$
   g_{F_1}(\xi,\xi)  
   \leq  c^{2n} \cdot   g_{F_2}(\xi,\xi).
$$
\end{proof}
 
\smallskip

 \begin{remark} { \label{rem5} 
Formula (\ref{eq.defgF3}) associates an ellipsoid in $V^*$ (the unit ball of the metric $g^*_{F}$) to an arbitrary convex body  $\Omega \subseteq V$.  This ellipsoid   is called the \emph{Binet ellipsoid} of $\Omega$ and appears in classical mechanics. }
The  unit ball $B \subset V$ of the metric $g_F$ is  the polar dual of the Binet ellipsoid. It is related to another  classic object: the  \emph{Legendre ellipsoid} $\mathcal{L}$ of $\Omega$ which is the unique ellipsoid such that
\begin{equation} \label{eq.inertia}
  \int_\mathcal{L} \theta^2(\xi) d\lambda(\xi)  =   \int_{\Omega}\theta^2(\xi) d\lambda(\xi)  
\end{equation}
for any $\theta \in V^*$. The Legendre ellipsoid $\mathcal{L}$ is related to the $g_F$-unit ball $B$ by the relation
$$
  \mathcal{L} = \left(\frac{\lambda(\Omega)}{\lambda(B)}\right)^{\frac{1}{n+2}} \cdot B
$$
which can be proved from equation (\ref{eq.duality1}) and property (a) of Proposition \ref{prop.binet2}.

\smallskip

The integral (\ref{eq.inertia}) is called the \emph{moment of inertia} of $\Omega$
in the codirection $\theta$. Thus \emph{the Legendre ellipsoid is the unique ellipsoid having the same
moment of inertia as $\Omega$ in all possible codirections} and  it has the following mechanical  interpretation: \emph{The motion of a homogenous rigid body   $\Omega$ which freely moves  in 3-space around the point $0$ and is subjected to no external force is dynamically equivalent to a similar motion of its Legendre ellipsoid}
(see  \cite{Legendre}).
\end{remark} 

\medskip

{
\paragraph{\bf Acknowledgement.}
This work benefited from discussions with  J. C. {\'A}lvarez Paiva, D. Burago, S. Deng, 
 S. Ivanov, A. Petrunin,  R. Schneider, Z. Shen,  J. Szilasi.
We also thank the Swiss National Science Foundation (200020-130107) and 
the Deutsche Forschungsgemeinschaft (GK 1523) for their 
financial support as well EPF Lausanne and FSU Jena for their hospitality.
}



\end{document}